\documentclass[11pt,letterpaper]{amsart}
\usepackage[utf8]{inputenc}
\usepackage{amsmath,amssymb,amsthm, mathtools, tikz-cd, verbatim}
\usepackage{diagbox}
\usepackage{geometry}
\usepackage{listings}
\usepackage{xcolor}

\title[The Arithmetic of Elliptic Pairs and An $\nu+1$-variable Artin Conjecture]{The Arithmetic of Elliptic Pairs and \\ An $\nu+1$-variable Artin Conjecture}
\author{Pranavkrishnan Ramakrishnan}
\date{September 2023}

\definecolor{codegreen}{rgb}{0,0.6,0}
\definecolor{codegray}{rgb}{0.5,0.5,0.5}
\definecolor{codepurple}{rgb}{0.58,0,0.82}
\definecolor{backcolour}{rgb}{0.95,0.95,0.92}

\lstdefinestyle{mystyle}{
  backgroundcolor=\color{backcolour}, commentstyle=\color{codegreen},
  keywordstyle=\color{magenta},
  numberstyle=\tiny\color{codegray},
  stringstyle=\color{codepurple},
  basicstyle=\ttfamily\footnotesize,
  breakatwhitespace=false,         
  breaklines=true,                 
  captionpos=b,                    
  keepspaces=true,                 
  numbers=left,                    
  numbersep=5pt,                  
  showspaces=false,                
  showstringspaces=false,
  showtabs=false,                  
  tabsize=2
}

\usepackage[
maxbibnames = 99,
backend=biber,
style=alphabetic,
]{biblatex}

\newcommand{\im}{\textup{im }}

\newcommand{\li}{\textup{li}}
\newcommand{\modp}{\textup{ mod }p}

\newcommand{\pseff}{\overline{\textup{Eff}}}
\newcommand{\Num}{\textup{Num}}

\newtheorem{theorem}{Theorem}[section]
\newtheorem*{theorem*}{Theorem}
\newtheorem{lemma}[theorem]{Lemma}
\newtheorem{corollary}[theorem]{Corollary}
\newtheorem{prop}[theorem]{Proposition}
\newtheorem*{corollary*}{Corollary}
\newtheorem*{lemma*}{Lemma}
\theoremstyle{definition}
\newtheorem{defin}{Definition}[section]
\newtheorem{remark}{Remark}[section]

\begin{document}
\maketitle
\begin{abstract}
    The theory of elliptic pairs, as investigated in a paper by Castravet, Laface, Tevelev, and Ugaglia, provides useful conditions to determine polyhedrality of the pseudo-effective cone, which give rise to interesting arithmetic questions when reducing the variety modulo $p$. In this paper, we examine one such case, namely the blow-up $X$ of 9 points in $\mathbb{P}^2$ lying on the nodal cubic, and study the density of primes $p$ for which the pseudo-effective cone of the reduction of $X$ modulo $p$ is polyhedral. This problem reduces to an analogue of Artin’s Conjecture on primitive roots like that investigated by Stephens and then Moree and Stevenhagen. As a result, we find that the density of such "polyhedral primes" hover around a higher analogue of the Stephens' Constant under the assumption of the Generalized Riemann Hypothesis. Finally, in order to determine a precise value for the density of polyhedral primes, we look at the containment of rank 8 root sublattices of $\mathbb{E}_8$. 
\end{abstract}
\section{Introduction}

\par An important invariant of a projective variety $X$ is its pseudo-effective cone $\overline{\textup{Eff}}(X)$. Taking an irreducible curve $C$ of arithmetic genus 1 on a smooth projective surface $X$ with $C^2=0$ forms a smooth elliptic pair $(C, X)$\cite[Def. 3.1]{https://doi.org/10.48550/arxiv.2009.14298}, which provides conditions to determine the polyhedrality of $\overline{\textup{Eff}}(X)$. In many cases, taking $X$ defined over $\mathbb{Q}, \pseff(X)$ is not polyhedral. However, when reduced by a prime $p$, the resulting surface $X_p$ may have a polyhedral pseudo-effective cone. Primes $p$ for which the pseudo-effective cone of $X_p$ are polyhedral are called polyhedral primes and the density of such primes, denoted $\delta(X)$, is an interesting question, and so the focus of this work. 
\par A specific case of this problem has been investigated in \cite[Sec. 4]{https://doi.org/10.48550/arxiv.2204.02971}: let $\Gamma\subset \mathbb{P}^2$ be a nodal cubic defined over $\mathbb{Q}$ and $z_1=1, \ldots, z_7=1, z_8=a, z_9=ba^{-1}\in \Gamma^{\textup{sm}}(\mathbb{Q})\cong \mathbb{Q}^*$ with $a, b\in\Gamma^{\textup{sm}}(\mathbb{Q})$ multiplicatively independent where 1 is a flex point. Define $X$ as the blow up of $\mathbb{P}^2$ at $z_i$'s with the blow up at $z_1,\ldots, z_7$ being infinitely near blow ups and $C$ as the proper transform of $\Gamma$ under the blow up. By \cite[Thm. 4.1]{https://doi.org/10.48550/arxiv.2204.02971}, the density of polyhedral primes in this case is equal to density of primes $p$ such that $b^2\modp\in\langle a\modp\rangle$, which we shall denote $\delta(a, b^2)$. The density of such primes is the focus of \cite{https://doi.org/10.48550/arxiv.math/9912250}. Under the assumption of the Generalised Riemann Hypothesis, $\delta(a,b^2)$ is a rational multiple of the Stephens' Constant $S$ \cite[Thm. 2]{https://doi.org/10.48550/arxiv.math/9912250}, which is defined as
\begin{equation}
    S = \prod_{p\textup{ prime}}\Big(1-\frac{p}{p^3-1}\Big)
\end{equation}
In this paper, we shall consider a different case: we shall let $z_i\in\mathbb{Q}^*_{>0}$ for distinct $z_i$ subject to the condition of \textit{multiplicative independence}, which is to say they form a basis of a $\mathbb{Z}$-submodule in $\mathbb{Q}^*$. Identifying the smooth rational points of $\Gamma$ with $\mathbb{Q}^*$ as before, we see that the restriction map $\textup{res}: \textup{Pic } X\rightarrow \textup{Pic } C$ induces an embedding $\overline{\textup{res}}$ of the root lattice $\mathbb{E}_8=C^\perp/\langle C\rangle$ into $\textup{Pic}^0 C/\langle \textup{res } C\rangle = \mathbb{Q}^*/\langle z_1\ldots z_9\rangle$, and likewise $\textup{res}_p: \textup{Pic } X_p\rightarrow \textup{Pic } C_p$ induces a linear map $\overline{\textup{res}_p}:\mathbb{E}_8\rightarrow \mathbb{F}_p^*/\langle z_1\ldots z_9 \textup{ mod } p\rangle$ for all but a finite number of primes $p$. By \cite[Cor. 3.18]{https://doi.org/10.48550/arxiv.2009.14298} $\overline{\textup{Eff}}(X_p)$ is polyhedral if and only if there exist 8 linearly independent roots $\beta_1,\ldots, \beta_8$ in $\mathbb{E}_8$ such that $\beta_i\in\ker\overline{\textup{res}_p}$ for all $i$. This means that the density of polyhedral primes is equal to the density of primes such that there exists 8 elements $b_1, \ldots, b_8\in \mathbb{Q}^*$ such that $\langle b_1, \ldots, b_8\mod z_1\ldots z_9\rangle$ corresponds to a rank 8 root lattice in $\mathbb{E}_8$ under the $\overline{\textup{res}}$ map and $b_1, \ldots, b_8 \in \langle z_1\ldots z_9\rangle \mod p$. To determine this, one must develop a theory of a 9 variable Artin conjecture, which shall be the focus of section 3, granted in a more general $\nu+1$ variable case for some $\nu\in\mathbb{N}$. Like with the usual Artin conjecture on primitive roots, a result is given under the assumption of the Generalized Riemann Hypothesis:
\begin{theorem}\label{ratmulnvar}
Let $a, b_1, \ldots, b_{\nu}\in\mathbb{Q}^*_{>0}$ such that they form a basis of a $\mathbb{Z}$-submodule in $\mathbb{Q}^*$, and define $\delta(a, b_1, \ldots, b_{\nu})$ be the density of the set
$$
\{p\textup{ prime}: b_1, \ldots, b_{\nu}\modp\in \langle a\modp\rangle\}.
$$
Under the assumption of the Generalised Riemann Hypothesis, $\delta(a, b_1, \ldots, b_{\nu})$ is a rational multiple of the Generalised Stephens' Constant $S^{(\nu)}$, defined as
\begin{equation}
    S^{(\nu)} = \prod_p\Big(1-\frac{p^\nu-1}{p-1}\Big(\frac{p}{p^{\nu+2}-1}\Big)\Big).
\end{equation}
\end{theorem}
\par The result of Theorem \ref{ratmulnvar} directly implies this geometric result: 
\begin{theorem}\label{mainresult}
Let $\Gamma$ be a nodal cubic and $z_1, \ldots, z_9$ be arbitrary points in $\Gamma^{\textup{sm}}(\mathbb{Q})\cong\mathbb{Q}^*$ such they form the basis of a $\mathbb{Z}$-submodule in $\mathbb{Q}^*$ where 1 is a flex point. Now let $X$ be the blow up of $\mathbb{P}^2$ at each $z_i$, and denote $\delta(X)$ to be the density of polyhedral primes of $X$. Under the assumption of the Generalised Riemann Hypothesis, the value of $\delta(X)$ is a rational multiple of $S^{(8)}$.
\end{theorem}
\subsection{Notation}\label{notationsection}
\par For the purposes of this work, for some fixed $\nu\in\mathbb{N}$, we shall denote $a, b_1, \ldots, b_\nu$ to be arbitrary nonzero rational numbers. For any $\xi\in\mathbb{R}_{>0}$, define $\xi\#=\prod_{0< p\leq \xi}p$. Furthermore for any $x\in\mathbb{Q}^*, x\neq \pm 1, $ we shall let $m_x\in \mathbb{N}$ be the largest natural number such that $x^{1/m_x}\in\mathbb{Q}$, and thus define $\widehat{x}$ such that $\widehat{x}^{m_x}=x$ and $c_x=\textup{ord}_2(m_x)$. Here $\textup{ord}_p(x)$ denotes the largest $k\in\mathbb{N}$ such that $p^k\mid x$. We shall let $[x, y]$ denote the least common multiple of $x$ and $y$ and let $[\alpha, x_i]_i$ denote the least common multiple of the set $\{\alpha\}\cup \{x_i\}_i$ (if it is just $[x_i]_i$ then $[x_i]_i$ equals the least common multiple of $\{x_i\}_i$.) $\Delta(x)$ shall denote the discriminant of $\mathbb{Q}(\sqrt{x})$, and $V=\langle \widehat{a}, \widehat{b_1}, \ldots, \widehat{b_\nu} \rangle/\langle \widehat{a}, \widehat{b_1}, \ldots, \widehat{b_\nu} \rangle^2\cong(\mathbb{Z}/2)^{\nu+1}$ for multiplicatively independent $a, b_1, \ldots, b_\nu$. Otherwise, in accordance with standard notation, $\zeta_k$ shall denote the primitive $k$-th root of unity and all instances of $p$ and $q$ shall refer to primes. 
\subsection{Main Results}
\begin{theorem}\label{O(disc)nvar}
Let $a, b_1, \ldots, b_\nu\in \mathbb{Q}^*_{>0}$ be multiplicatively independent such that 
$$
\textup{tor }\mathbb{Q}^*/\langle -1, a, b_1, \ldots, b_\nu\rangle\cong \mathbb{Z}/m_a\oplus\mathbb{Z}/m_{b_1}\oplus\ldots\oplus\mathbb{Z}/m_{b_\nu}
$$
is generated by $\widehat{a}, \widehat{b_1}, \ldots, \widehat{b_\nu}$. Under the assumption of the Generalised Riemann Hypothesis:
$$
\delta(a, b_1,\ldots, b_\nu) = \sum_{\substack{d\mid m_a\\ d_1\mid m_{b_1}\\ \cdots\\ d_\nu\mid m_{b_\nu}}}\Big(\varphi(d)\prod_{h=1}^\nu \varphi(d_h)\Big)S^{(\nu)}_{[d_h]_h, d}+\sum_{\substack{x\in V\\ x\neq 1}} O\left(\frac{\log\Delta(\Hat{x})}{\Delta(\widehat{x})^2}\right).
$$
where
$$
S_{m,n}^{(\nu)} = \frac{S^{(\nu)}}{[m,n]^{\nu+2}}\prod_{p\mid \frac{n}{(m,n)}}\frac{-p^{\nu+3}(p^\nu-1)}{p^{\nu+3}-p^{\nu+2}-p^{\nu+1}+1}\prod_{p\mid \frac{m}{(m,n)}}\frac{p^{\nu+1}(p^2-1)}{p^{\nu+3}-p^{\nu+2}-p^{\nu+1}+1}.
$$
\end{theorem}\label{tablemainthm}
A precise formula for $\delta(a, b_1, \ldots, b_\nu)$ is provided in Section 3. 
\begin{theorem}
    Let $\Lambda\subseteq\mathbb{E}_8$ be a rank 8 root lattice, and let $\overline{\delta}(\Lambda)$ denote the density of primes such that $\ker\overline{\textup{res}_p} = \Lambda$. Under the assumption of the Generalized Riemann Hypothesis, the value of $\overline{\delta}(\Lambda)$ as $\Delta(\widehat{z_i})\rightarrow\infty$ is given in the last column of the following table:
    \begin{table}[ht]
        \centering
        \begin{tabular}{|c|c|c|c|c|}
    \hline
        $R$ & $\mathbb{E}_8/R$ & $H$ & $\#\{\Lambda\subset\mathbb{E}_8\mid \Lambda\cong R, \Pi/\Lambda\cong H\}$ & $\overline{\delta}(\Lambda)/S^{(8)},\Lambda\cong R, \Pi/\Lambda\cong H$\\ \hline
        $\mathbb{A}_8$ & $\mathbb{Z}/3$ & $\mathbb{Z}/9$ & \textup{648} & \textup{1/363210399} \\ \cline{3-5}
        ~ & ~ & $(\mathbb{Z}/3)^2$ & \textup{312} & \textup{1/6151} \\ \hline
        $\mathbb{D}_8$ & $\mathbb{Z}/2$ & $\mathbb{Z}/6$ & \textup{135} & \textup{9227/3155463} \\ \hline
        $\mathbb{E}_7\oplus\mathbb{A}_1$ & $\mathbb{Z}/2$ & $\mathbb{Z}/6$ & \textup{120} & \textup{9227/3155463}\\ \hline
        $\mathbb{A}_5\oplus\mathbb{A}_2\oplus\mathbb{A}_1$ & $\mathbb{Z}/6$ & $\mathbb{Z}/18$ & \textup{27216} & \textup{2187/1960736456704} \\ \cline{3-5}
        ~ & ~ & $\mathbb{Z}/6\oplus\mathbb{Z}/3$ & \textup{13104} & \textup{1/2103642} \\ \hline
        $\mathbb{A}_4^{\oplus 2}$ & $\mathbb{Z}/5$ & $\mathbb{Z}/15$ & \textup{12096} & \textup{9227/17832794670} \\ \hline
        $\mathbb{E}_6\oplus\mathbb{A}_2$ & $\mathbb{Z}/3$ & $\mathbb{Z}/9$ & \textup{756} & \textup{1/363210399} \\ \cline{3-5}
        ~  & ~ & $(\mathbb{Z}/3)^2$ & \textup{364} & \textup{1/6151} \\ \hline
        $\mathbb{A}_7\oplus\mathbb{A}_1$ & $\mathbb{Z}/4$ & $\mathbb{Z}/12$ & \textup{4320} & \textup{9227/1615597056} \\ \hline
        $\mathbb{D}_6\oplus\mathbb{A}_1^{\oplus 2}$ & $(\mathbb{Z}/2)^2$ & $\mathbb{Z}/6\oplus\mathbb{Z}/2$ & \textup{540} & \textup{0} \\ \hline
        $\mathbb{D}_5\oplus\mathbb{A}_3$ & $\mathbb{Z}/4$ & $\mathbb{Z}/12$ & \textup{7560} & \textup{9227/1615597056} \\ \hline
        $\mathbb{D}_4^{\oplus 2}$ & $(\mathbb{Z}/2)^2$ & $\mathbb{Z}/6\oplus\mathbb{Z}/2$ & \textup{1575} & \textup{0} \\ \hline
        $(\mathbb{A}_3\oplus\mathbb{A}_1)^{\oplus 2}$ & $\mathbb{Z}/4\oplus\mathbb{Z}/2$ & $\mathbb{Z}/6\oplus\mathbb{Z}/4$ & \textup{37800} & \textup{0} \\ \hline
        $\mathbb{A}_2^{\oplus 4}$ & $(\mathbb{Z}/3)^2$ & $\mathbb{Z}/9\oplus\mathbb{Z}/3$ & \textup{10080} & \textup{0} \\ \cline{3-4}
        ~ & ~ & $(\mathbb{Z}/3)^3$ & \textup{1120} & \\ \hline
        $\mathbb{E}_8$ & $1$ & $\mathbb{Z}/3$ & \textup{1} & \textup{73813/73812} \\ \hline
    \end{tabular}
        \caption{Values of $\overline{\delta}(\Lambda)$ for root lattices $\Lambda$}
        \label{tab:density_table}
    \end{table}
\newpage
where $\Pi$ denotes the lattice $\langle z_1, \ldots, z_9\rangle/\langle z_1\ldots z_9\rangle$. As a result, 
$$
\lim_{\Delta(\widehat{z_i})\rightarrow\infty, \forall i}\delta(X)=83568208560360063877/43166735003229880320S^{(8)}\approx 0.72609882811\ldots
$$ 
where $\delta(X)$ is the density of polyhedral primes and
$$
S^{(8)}\approx 0.375062673164990163033863645604\ldots
$$
\end{theorem}
It is worth noting that the structure of the singular fibres of rational elliptic fibrations in the case of $\mathbb{A}_2^{\oplus 4}, \mathbb{D}_6\oplus\mathbb{A}_1^{\oplus 2},$ and $\mathbb{D}_4^{\oplus 2}$ do not permit a lone nodal cubic, which gives us a reason independent of the Generalized Riemann Hypothesis that their density equals 0. 
\subsection{Acknowledgements}
I can not express enough gratitude for my research advisor, Jenia Tevelev, for his guidance, feedback, and incredible patience in getting this work to where it is. I would also like to extend my gratitude to Tom Weston for useful discussions, and Elizabeth Pratt for forming the bedrock of my work. Finally, I am also incredibly thankful to Anibal Aravena and Elias Sink for help with Section 3. This project has received the support of the NSF grant DMS-2101726 (PI Jenia Tevelev.) 

\section{Polyhedrality and the Theory of Elliptic Pairs}\label{polyhedsec}
\par We shall start by recalling the definitions presented in \cite[Sec. 2]{https://doi.org/10.48550/arxiv.2009.14298}. Namely, for an algebraically closed field of arbitrary characteristic $k$, let $X$ be a normal protective irreducible variety over $k$ and $\textup{Pic}(X)$ its Picard group. Moreover, $\sim$ shall denote the linear equivalence of divisors and $\equiv$ to be numerical equivalence, which in the case of Cartier divisors means that $D_1\equiv D_2$ if and only if $D_1\cdot C = D_2\cdot C$ for all curves $C\subseteq X$. We can define $\Num^1(X):= \textup{Pic}(X)/\equiv$, and let $\Num^1(X)_\mathbb{R}=\Num^1(X)\otimes_{\mathbb{Z}}\mathbb{R}$ and $\Num^1(X)_\mathbb{Q}=\Num^1(X)\otimes_{\mathbb{Z}}\mathbb{Q}$. The \textit{effective cone} of $X$ is the cone generated by numerical classes of effective Cartier divisors, and the \textit{pseudo-effective cone} $\pseff(X)\subseteq\Num^1(X)_\mathbb{R}$ is its closure. Such a cone allows for the condition of \textit{polyhedrality}: 
\begin{defin}
A convex cone $\mathcal{C}\subseteq \mathbb{R}^s$ is called \textit{polyhedral} if there exist finitely many vectors $v_1, \ldots, v_s$ such that $\mathcal{C}=\sum\mathbb{R}_{>0}v_i$. It is called \textit{rational polyhedral} if $v_i\in\mathbb{Q}^s$ for all $i$.
\end{defin}
\begin{defin}\textup{\cite[Def. 3.1]{https://doi.org/10.48550/arxiv.2009.14298}}
    An elliptic pair $(C, X)$ consists of a projective rational surface $X$ with log terminal singularities and an irreducible curve $C \subseteq X$, of arithmetic genus one, disjoint from the singular locus of $X$ and such that $C^2 = 0$. Let $C^\perp \subseteq \textup{Cl}(X)$ be the orthogonal complement to $C$. We define the restriction map
    $$
    \textup{res} : C^{\perp}\rightarrow\textup{Pic}^0(C), \hspace*{5mm} D\mapsto \mathcal{O}(D)|_C.
    $$
    Since $K\cdot C= 0$ by adjunction, we can also define the \textit{reduced restriction map}
    $$
    \overline{\textup{res}} : \textup{Cl}_0(X) := C^\perp/\langle K\rangle \rightarrow \textup{Pic}^0(C)/\langle \textup{res}(K)\rangle
    $$
\end{defin}
Now, in order to get a condition for polyhedrality we take the following setup: let $\Gamma$ be a nodal cubic in $\mathbb{P}^2$ defined over $\mathbb{Q}$, containing $[0:1:0]$ as a flex point, and take $X$ as the blowup of $\mathbb{P}^2$ at nine rational smooth points in $\Gamma$ and $C$ to be the proper transform of $\Gamma$. The intersection pairing on $X$ allows us to identify $C^\perp/\langle C\rangle$ with the root lattice $\mathbb{E}_8$. From this we are given an elliptic pair $(C, X)$.  Since $K\sim -C$, we see that 
$$
\overline{\textup{res}} : C^\perp/\langle C\rangle = \mathbb{E}_8 \rightarrow \textup{Pic}^0(C)/\langle \textup{res}(\textup{C})\rangle
$$
and since $\textup{Pic}^0(C)\cong \mathbb{Q}^*$ and $\textup{res}(\textup{C})=z_1\ldots z_9$, we can view $\overline{\textup{res}}$ as the map
$$
\overline{\textup{res}} : \mathbb{E}_8 \rightarrow \mathbb{Q}^*/\langle z_1 \ldots z_9\rangle.
$$
We assume that $z_1, \ldots, z_9$ have infinite order over $\mathbb{Q}^*$, and so over $\mathbb{Q}$ we see that $\textup{res}(C)$ has infinite order in $\textup{Pic}^0(C)$, and so by \cite[Lem. 3.3]{https://doi.org/10.48550/arxiv.2009.14298} we see that $\overline{\textup{Eff}}(X)$ is not polyhedral. However, if we take the geometric fibers of $(C, X)$ over some prime $p$ we see that in fact $\textup{res}(C)$ has finite order in $\textup{Pic}^0(C)$. By \cite[Thm. 3.8]{https://doi.org/10.48550/arxiv.2009.14298} the condition that $\overline{\textup{Eff}}(X_p)$ is polyhedral is equivalent to the condition that $\ker(\overline{\textup{res}_p})$ contains 8 linearly independent roots, where 
$$
\overline{\textup{res}_p} : \mathbb{E}_8 \rightarrow \mathbb{F}_p^*/\langle z_1 \ldots z_9 \textup{ mod } p \rangle.
$$
This means that the density of polyhedral primes is the density of primes such that $\ker(\overline{\textup{res}_p})$ contains 8 linearly independent roots of $\mathbb{E}_8$, or equivalently, the density of primes such that there are 8 linearly independent roots $\beta_1, \ldots, \beta_8\in \mathbb{E}_8$ such that $\overline{\textup{res}_p}(\beta_i)\in \langle z_1\ldots z_9 \textup{ mod } p\rangle$ for all $i$.

\section{A Multivariable Artin's Conjecture}
\par As in the setup of Theorem \ref{O(disc)nvar}, let $a, b_1, \ldots, b_{\nu}\in\mathbb{Q}^*_{>0}$ such that they form a basis of a $\mathbb{Z}$-submodule in $\mathbb{Q}^*$ with $\textup{tor }\mathbb{Q}^*/\langle -1, a, b_1, \ldots, b_\nu\rangle\cong\mathbb{Z}/m_a\oplus \mathbb{Z}/m_{b_1}\oplus\ldots\oplus \mathbb{Z}/m_{b_\nu}$ generated by $\widehat{a}, \widehat{b_1}, \ldots, \widehat{b_\nu}$. With this we may define $\delta(a, b_1, \ldots, b_{\nu})$ be the density of the set
$$
\{p\textup{ prime}: b_1, \ldots, b_{\nu}\modp\in \langle a\modp\rangle\}
$$
In this section we shall show the existence of this function and provide a formula under the assumption of the Generalised Riemann Hypothesis for Dedekind Zeta functions. 
\par In \cite{https://doi.org/10.48550/arxiv.math/9912250}, in order to obtain a density formula in the two variable case $\nu=1$, the following sum is considered:
$$
\delta(a, b)= \sum_{i=1}^\infty\delta_i(a, b),
$$
where $\delta_i(a, b)$ denotes the density of primes $p$ such that $[\mathbb{F}_p^*:\langle a\rangle\modp]=i$ and $i\mid [\mathbb{F}_p^*:\langle b\rangle\modp]$. Under the assumption of the Generalised Riemann Hypothesis,
$$
\delta_i(a, b) = \sum_{j=1}^\infty \frac{\mu(j)}{[F_{i, j}:\mathbb{Q}]}=\prod_p \Big(1-\frac{1}{[F_{i, p}:\mathbb{Q}]}\Big).
$$
Here $F_{i, j}=\mathbb{Q}(\zeta_{ij}, a^{1/ij}, b^{1/i})$. This permits a natural generalisation as for an $\nu+1$ variable case, we may define $F_{i, j}=\mathbb{Q}(\zeta_{ij}, a^{1/ij}, b_1^{1/i}, \ldots, b_\nu^{1/i})$, using which we get the following theorem: 
\begin{theorem}\label{maindenthm}
Under the assumption of the Generalised Riemann Hypothesis, 
\begin{equation}\label{mainden}
    \delta(a, b_1, \ldots, b_{\nu}) = \sum_{i=1}^\infty\sum_{j=1}^\infty \frac{\mu(j)}{[F_{i, j}:\mathbb{Q}]}.
\end{equation}
\end{theorem}
\par $[F_{i,j}:\mathbb{Q}]$ differs from $i^{\nu+1}j\varphi(ij)$ by some factor $f_{i,j}$, which is referred to as the "degree loss" in \cite{https://doi.org/10.48550/arxiv.math/9912250}. Basing our methods on \cite[Sec. 3]{https://doi.org/10.48550/arxiv.math/9912250}, this factor can be determined via Kummer Theory, which shall be the focus of Section \ref{radextsec}. The consequence of this is the following result:
\begin{theorem}\label{bigeqden}
Under the assumption of the Generalised Riemann Hypothesis,

\begin{equation*}
    \delta(a, b_1,\ldots, b_\nu) = \sum_{\substack{d\mid m_a\\ d_1\mid m_{b_1}\\ \cdots\\ d_\nu\mid m_{b_\nu}}}\Big(\varphi(d)\prod_{h=1}^\nu \varphi(d_h)\Big)\left(S_{[d_h]_h, d}+\sum_{x\in \widehat{a}V}S_{[2^{C_x},d_h]_h, [2^{c_a+1}, d, \Delta(\widehat{x})]}+\sum_{\substack{x\in V/\langle\widehat{a}\rangle\\ x\neq 1}}S_{[2^{C_x}, d_h]_h, [ d, \Delta(\widehat{x})]} \right)
\end{equation*}
where for $x=\widehat{a}^{n_0}\prod_{h=1}^{\nu}\widehat{b_h}^{n_h}$,
$$
C_x = \max_{h\in \{1,2,\ldots, \nu\}}\{(c_{b_h}+1)n_h\}.
$$
Here we use notation from Section \ref{notationsection} and the Introduction.
\end{theorem}
As a Corollary of Theorem \ref{bigeqden}, we can prove Theorems \ref{ratmulnvar} and \ref{O(disc)nvar}. 
\subsection{Results of \cite{STEPHENS1976313}}
\par We shall start by recounting the relevant results of Stephens, and Moree and Stevenhagen. In \cite{STEPHENS1976313}, the main theorem concerns the function $\delta_i(a)$ which is the density of the set
$$
\{p\textup{ prime }\mid [\mathbb{F}_p^*:\langle a\rangle]=i\}.
$$
In order to compute this, $R(p, q, i)$ is defined to be the condition that $a^{\frac{p-1}{qi}}\equiv 1\modp$ (and in particular that $qi\mid p-1 $). With this, the following functions are defined:
\begin{center}
    \begin{tabular}{|m{2cm}|m{9cm}|}
    \hline
    $N_{a, i}(x)$ & the number of primes $p\leq x$ such that $a^{\frac{p-1}{i}}\equiv 1\modp$ and $R(p,q,i)$ is not satisfied for all primes $q\mid (p-1)/i$. \\
    \hline
    $N_{a, i}(x, \eta_1)$ & the number of primes $p\leq x$ such that $a^{\frac{p-1}{i}}\equiv 1\modp$ and $R(p,q,i)$ is not satisfied for all primes $q\mid (p-1)/i, q\leq \eta_1$.\\
    \hline
    $M_{a,i}(x, \eta_1, \eta_2)$ & the number of primes $p\leq x$ such that $R(p, q,i)$ is satisfied for some prime $q\mid (p-1)/i, \eta_1< q\leq \eta_2$.\\
    \hline
    $P_{a, i}(x, k), k$ squarefree & the number of primes $p\leq x$ such that $R(p, q,i)$ is true for all primes $q\mid k$.\\
    \hline
\end{tabular}
\end{center}
\par Setting $\xi_1=\frac{1}{6}\log x, \hspace*{3mm} \xi_2=x^{1/2}/\log^5 x, \hspace*{3mm} \xi_3=x^{1/2}\log x$, the following formulas are established for $x\gg \xi_1, \xi_2, \xi_3$:
\begin{equation}
    M_{a, i}(x, \xi_2, \xi_3) = O\Big(\frac{x\log\log x}{\log^2x}\Big)\textup{\indent\cite[Lem. 3]{STEPHENS1976313}},
\end{equation}
\begin{equation}
    M_{a, i}(x, \xi_3, (x-1)/i) = O\Big(\frac{x}{\log^2x}\Big)\textup{\indent\cite[Lem. 4]{STEPHENS1976313}},
\end{equation}
and, under the assumption of the Generalised Riemann Hypothesis,
\begin{equation}
    M_{a, i}(x, \xi_1, \xi_2) = O\Big(\frac{x\log\log x}{\log^2x}\Big)\textup{\indent\cite[Lem. 8]{STEPHENS1976313}}.
\end{equation}
Combining all three
$$
M_{a, i}(x, \xi_1, (x-1)/i)\leq M_{a, i}(x, \xi_1, \xi_2)+M_{a, i}(x, \xi_2, \xi_3)+M_{a, i}(x, \xi_3, (x-1)/i),
$$
the following is acquired:
$$
M_{a, i}(x, \xi_1, (x-1)/i)= O\Big(\frac{x\log\log x}{\log^2x}\Big)+O\Big(\frac{x\log\log x}{\log^2x}\Big)+O\Big(\frac{x}{\log^2x}\Big).
$$
Notable is that $M_{a, i}(x, \xi_1, x-1)=M_{a, i}(x, \xi_1, (x-1)/i)$ since the values of all possible $q$'s will be lesser or equal to $(x-1)/i$. It now follows from this fact that
\begin{equation}\label{O(M)}
    O(M_{a, i}(x, \xi_1, x-1))= O\Big(\frac{x\log\log x}{\log^2x}\Big)
\end{equation}
under the assumption of the Generalised Riemann Hypothesis. Next, the following is established:
$$
N_{a, i}(x) = N_{a, i}(x, \xi_1) + O(M_{a,i}(x, \xi_1, x-1))
$$
and by the inclusion exclusion principle the following is also established:
$$
N_{a, i}(x, \xi_1) = \sum_{j\mid \xi_1\#}\mu(j)P_{a, i}(x, j).
$$
Then the following important equality is shown under the assumption of the Riemann Hypothesis for Dedekind zeta functions over fields of the form $F_i=\mathbb{Q}(\zeta_i, a^{1/i})$:
$$
P_{a, i}(x, j) = \frac{\li(x)}{[F_i:\mathbb{Q}]} + O(\sqrt{x}\log x).
$$
Finally, combining all these terms one gets
$$
N_{a,i}(x)= \Big(\sum_{j\mid \xi_1\#}\frac{\mu(j)}{[F_{i}:\mathbb{Q}]}\Big)\li(x) + O\Big(\frac{x\log\log x}{\log^2x}\Big),
$$
from which the main results are derived. Notable is the calculation of $\delta(a, b)$, which is the density of primes $p$ such that $b\modp\in\langle a\rangle\modp$. Under the assumption of the Riemann Hypothesis for Dedekind zeta functions over fields of form $F_{i,j}=\mathbb{Q}(\zeta_{ij}, a^{1/ij}, b^{1/i})$, it is proved that
$$
\delta(a, b)= \sum_{i=1}^\infty\sum_{j=1}^\infty \frac{\mu(j)}{[F_{i, j}:\mathbb{Q}]}.
$$
\par However, we can note that the restrictive conditions imposed on $a$ and $b$ as in \cite[Thm. 3]{STEPHENS1976313} outside of multiplicative independence are largely unnecessary and through the analysis of \cite{https://doi.org/10.48550/arxiv.math/9912250} are not required to get an exact value for $\delta(a,b)$. 
\subsection{Generalization of \cite{STEPHENS1976313} to an $\nu+1$ variable case}
\par Let $A=\{a, b_1, \ldots, b_\nu\}$, with $a, b_1, \ldots, b_\nu$ defined as in Theorem \ref{O(disc)nvar}, and $N_{A, i}(x)$ be the number of primes $p$ less than $x$ such that $[\mathbb{F}_p^*: \langle a\rangle\modp]=i$ and $b_1,\ldots, b_\nu\modp\in \langle a\rangle\modp$. In order to get (\ref{mainden}), we shall establish the following theorem:
\begin{theorem}\label{mainanalytictheorem}
Assuming the Generalized Riemann Hypothesis, the following formula holds:
\begin{equation}\label{Naieq}
    N_{A,i}(x)= \Big(\sum_{j\mid \xi_1\#}\frac{\mu(j)}{[F_{i, j}:\mathbb{Q}]}\Big)\li(x) + O\Big(\frac{x\log\log x}{\log^2x}\Big),
\end{equation}
where $\xi_1=\frac{1}{6}\log x$, $\xi_1\#=\prod_{p \leq \xi_1}p$,  $F_{i,j}=\mathbb{Q}(a^{1/ij}, b_1^{1/j}, \ldots, b_\nu^{1/j}),$ and the constant in the big $O$ depends on $i$.
\end{theorem}
Setting $x\rightarrow\infty$  in (\ref{Naieq}) implies Theorem \ref{maindenthm} of this section as corollary. 
\par Define $R(p,i)$ to be the condition that $a^{\frac{p-1}{i}}, b_1^{\frac{p-1}{i}}, \ldots, b_\nu^{\frac{p-1}{i}}\equiv 1(\textup{mod }p)$. By the definition of $N_{A, i}(x)$, we see that it is equal to the number of primes $p\leq x$ such that $R(p,i)$ is satisfied and $a^{\frac{p-1}{qi}}\not\equiv 1(\textup{mod }p)$ for all primes $q\mid (p-1)/i$. We also define the following: 
\begin{center}
    \begin{tabular}{|c|m{9cm}|}
        \hline
        $N_{A,i}(x, \eta)$ & The number of primes $p\leq x$ such that $R(p,i)$ is satisfied and $a^{\frac{p-1}{qi}}\not\equiv 1(\textup{mod }p)$ for all primes $q\mid (p-1)/i$ such that $q\leq \eta$. \\
        \hline
        $M_{A, i}(x, \eta_1, \eta_2)$ & The number of primes $p\leq x$ that satisfy the properties that $a^{\frac{p-1}{qi}}\equiv 1(\textup{mod }p)$ for some primes $q\mid (p-1)/i, \eta_1\leq q\leq \eta_2$ (defined as in \cite{STEPHENS1976313}).\\
        \hline
        $P_{A, i}(x, k), k$ square-free & The number of primes $p\leq x$ such that $R(p,i)$ is satisfied and $a^{\frac{p-1}{qi}}\equiv 1(\textup{mod }p)$ for all prime $q\mid k$\\
        \hline
    \end{tabular}
\end{center}
\par Notable is that $P_{A, 1}(x, 1)=\pi(x)$, that, in general, $P_{A, i}(x, 1)$ equals the number of primes less than $x$ such that all elements of $A$ are $i$-th residues, and that $N_{A,i}(x)=N_{A,i}(x, (x-1)/i).$\par For sufficiently large $x$ we see that 
$$
N_{A, i}(x) \leq N_{A, i}(x, \xi_1),
$$
as the condition on primes for $N_{A, i}(x)$ is stronger than the condition for $N_{A, i}(x, \eta)$ and that 
$$
N_{A, i}(x) \geq N_{A, i}(x, \xi_1) - M_{A, i}(x, \xi_1, x-1).
$$
This gives us that 
$$
    N_{A, i}(x) = N_{A, i}(x, \xi_1) + O(M_{A, i}(x, \xi_1, x-1)).
$$
By (\ref{O(M)})
\begin{equation}\label{initOform}
    N_{A, i}(x) = N_{A, i}(x, \xi_1) + O\Big(\frac{x\log\log x}{\log^2x}\Big).
\end{equation}
Now examining $N_{A, i}(x, \xi_1)$ we see by the inclusion exclusion principle that
\begin{equation}\label{inexeq}
    N_{A, i}(x, \xi_1) = \sum_{j\mid \xi_1\#}\mu(j)P_{A, i}(x, j).
\end{equation}
We shall now show the following lemmas:
\begin{lemma}
$$
\frac{\log|\Delta_{F_{i, j}/\mathbb{Q}}|}{[F_{i, j}:\mathbb{Q}]} \leq \lambda\log j + C(i),
$$
for some finite fixed $\lambda$ and $C(i)$ depending only on $i$
\end{lemma}
\begin{proof}\label{Ologj}
By the main theorem of \cite{ToyamaComposedField} we see that 
$$
\Delta_{F_{i,j}/\mathbb{Q}} \mid \Delta_{\mathbb{Q}(\zeta_{ij}, a^{1/ij})/\mathbb{Q}}^{[F_{i, j}:\mathbb{Q}(\zeta_{ij}, a^{1/ij})]}\Delta_{\mathbb{Q}(b_1^{1/i},\ldots, b_\nu^{1/i})/\mathbb{Q}}^{[F_{i, j}:\mathbb{Q}(b_1^{1/i},\ldots, b_\nu^{1/i})]}.
$$
Now noting that $[F_{i, j}:\mathbb{Q}(\zeta_{ij}, a^{1/ij})]\leq [\mathbb{Q}(b_1^{1/i},\ldots, b_\nu^{1/i}):\mathbb{Q}]=i^\nu$ and that $[F_{i, j}:\mathbb{Q}(b_1^{1/i},\ldots, b_\nu^{1/i})]\leq [F_{i,j}:\mathbb{Q}]$ we see that 
\begin{align}
    |\Delta_{F_{i, j}/\mathbb{Q}}| &\leq |\Delta_{\mathbb{Q}(\zeta_{ij}, a^{1/ij})/\mathbb{Q}}^{i^\nu}\Delta_{\mathbb{Q}(b_1^{1/i},\ldots, b_\nu^{1/i})/\mathbb{Q}}^{[F_{i, j}:\mathbb{Q}]}|\\
    |\Delta_{F_{i, j}/\mathbb{Q}}|^{\frac{1}{[F_{i, j}:\mathbb{Q}]}} &\leq |\Delta_{\mathbb{Q}(\zeta_{ij}, a^{1/ij})/\mathbb{Q}}^{{\frac{i^\nu}{[F_{i, j}:\mathbb{Q}]}}}||\Delta_{\mathbb{Q}(b_1^{1/i},\ldots, b_\nu^{1/i})/\mathbb{Q}}|. \label{discmeq}
\end{align}
Now to bound $|\Delta_{\mathbb{Q}(\zeta_{ij}, a^{1/ij})/\mathbb{Q}}|$ we see from the main result in \cite{ToyamaComposedField} that 
\begin{align*}
    |\Delta_{\mathbb{Q}(\zeta_{ij}, a^{1/ij})/\mathbb{Q}}| &\leq  |\Delta_{\mathbb{Q}(\zeta_{ij})/\mathbb{Q}}|^{ij}|\Delta_{\mathbb{Q}(a^{1/ij})/\mathbb{Q}}|^{\varphi(ij)}\\
    \intertext{We know by \cite[Lem. 2.6]{HANDISC} that $\Delta_{\mathbb{Q}(\zeta_{ij})/\mathbb{Q}}\leq ij^{\varphi(ij)}$. Additionally we know that $\Delta_{\mathbb{Q}(a^{1/ij})/\mathbb{Q}}\mid \Delta_{x^{ij}-a}$, where $\Delta_{x^{ij}-a}$ denotes the discriminant of the splitting field of $x^{ij}-a$, and thus $\Delta_{\mathbb{Q}(a^{1/ij})/\mathbb{Q}}\mid (ij)^{ij}a^{ij}$ by \cite[Eqn. 2.3]{HANDISC}. This means that }
    |\Delta_{\mathbb{Q}(\zeta_{ij}, a^{1/ij})/\mathbb{Q}}| &\leq |(ij)^2a|^{ij\varphi(ij)}.
\end{align*}
Substituting this into Equation \ref{discmeq} we get that
\begin{align*}
    |\Delta_{F_{i, j}/\mathbb{Q}}|^{\frac{1}{[F_{i, j}:\mathbb{Q}]}} &\leq |(ij)^{2}a|^{{\frac{i^{\nu+1}j\varphi(ij)}{[F_{i, j}:\mathbb{Q}]}}}||\Delta_{\mathbb{Q}(b_1^{1/i},\ldots, b_\nu^{1/i})/\mathbb{Q}}|.
    \intertext{ Noting that $i^{\nu+1}j\varphi(ij)/[F_{i,j}:\mathbb{Q}]$ is bounded by $\lambda=2^{\nu+2}m_a\cdot \prod_{h=1}^\nu m_{b_h}$ by Corollary \ref{boundondeg}, we see that}
    &\leq |(ij)\sqrt{a}|^{\lambda}||\Delta_{\mathbb{Q}(b_1^{1/i},\ldots, b_\nu^{1/i})/\mathbb{Q}}|.
\end{align*}
This means that setting $C(i)=\log|(i\sqrt{a})^{\lambda}\Delta_{\mathbb{Q}(b_1^{1/i},\ldots, b_\nu^{1/i})/\mathbb{Q}}|$ and taking the logarithm on both sides we get that 
$$
\frac{\log|\Delta_{F_{i, j}/\mathbb{Q}}|}{[F_{i, j}:\mathbb{Q}]} \leq \lambda\log j + C(i),
$$
as desired.
\end{proof}
\begin{lemma}\label{POform}
Under the assumption of the Generalized Riemann Hypothesis, the following equation holds.
$$
\Big|P_{A, i}(x, j) - \frac{\li(x)}{[F_{i, j}: \mathbb{Q}]}\Big|\leq c_1\sqrt{x}\Big(\log x+\lambda\log j+C(i)\Big).
$$
for any squarefree $j$ and fixed absolute $c_1>0$. 
\end{lemma}
\begin{proof} 
 For all primes $p$ and any prime $q$ the condition that $a^{\frac{p-1}{i}}, b_1^{\frac{p-1}{i}}, \ldots, b_\nu^{\frac{p-1}{i}}, a^{\frac{p-1}{qi}}\equiv 1(\textup{mod }p)$ is equivalent to the condition that $a^{1/iq}, b_h^{1/i}\in\mathbb{F}^*_p$ for all $h$ and $p\equiv 1 \textup{ mod }qi$. As $p\equiv 1 \textup{ mod }iq$ is equivalent to the condition that $\zeta_{iq}\in\mathbb{F}^*_p$, it follows that $a^{1/iq}, b_h^{1/i}\in\mathbb{F}^*_p$ for all $h$ and $p\equiv 1 \textup{ mod }qi$ is equivalent to the condition that the minimal polynomial of $F_{i, j}/\mathbb{Q}$ splits completely in $\mathbb{F}_p^*$. And so the condition that $a^{\frac{p-1}{i}}, b_1^{\frac{p-1}{i}}, \ldots, b_\nu^{\frac{p-1}{i}}, a^{\frac{p-1}{qi}}\equiv 1(\textup{mod }p)$ is equivalent to the condition $p$ splits completely in $F_{i, q}$ for all but finitely many primes $p$, which likewise means that for any square-free $j$ the condition that $a^{\frac{p-1}{i}}, b_1^{\frac{p-1}{i}}, \ldots, b_\nu^{\frac{p-1}{i}}, a^{\frac{p-1}{qi}}\equiv 1(\textup{mod }p)$ for $q\mid j$ is equivalent to the condition that $p$ splits completely in $F_{i, j}$ for all but finitely many $p$. Therefore by \cite[Thm. 1.1]{LO75}, this means that 
$$
    \Big|P_{A, i}(x, j)-\frac{\li(x)}{[F_{i, j}: \mathbb{Q}]}\Big| \leq \frac{c_1}{[F_{i, j}: \mathbb{Q}]}(\sqrt{x}([F_{i, j}:\mathbb{Q}]\log x + \log |\Delta_{F_{i, j/\mathbb{Q}}}|)),
$$
for some effectively computable absolute constant $c_1$. Finally, by Lemma 2.5 we see that 
$$
\Big|P_{A, i}(x, j) - \frac{\li(x)}{[F_{i, j}: \mathbb{Q}]}\Big|\leq c_1\sqrt{x}\Big(\log x+\lambda\log j+C(i)\Big),
$$
as desired. 
\end{proof}
\begin{lemma}\label{O(sqrtlogx)}
\begin{equation*}
    \sum_{j\mid \xi_1\#} \sqrt{x}\log jx = O\Big(\frac{x}{\log^2x}\Big).
\end{equation*}
\end{lemma}
\begin{proof}
We first see that $j\leq \prod_{p\leq \xi_1 }p\leq e^{2\xi_1}=x^{1/3}$. This means that $\sqrt{x}\log jx\leq \sqrt{x}\log x^{4/3}=\frac{4}{3}\sqrt{x}\log x$. It follows that 
$$
    \sum_{j\mid \xi_1\#} \sqrt{x}\log jx = \sum_{\substack{j\mid \xi_1\#\\ j\leq x^{1/3}}} \frac{4}{3}\sqrt{x}\log x= \frac{4}{3}x^{5/6}\log x= O(x^{5/6}\log x) = O\Big(\frac{x}{\log^2 x}\Big).
$$
\end{proof}
With this we may show the main theorem:
\begin{proof}[Proof of Theorem \ref{mainanalytictheorem}]
From Lemma \ref{POform}, we see that 
$$
\Big| \sum_{j\mid \xi_1\#}\mu(j)P_{A, i}(x, j) - \sum_{j\mid \xi_1\#}\mu(j)\frac{\li(x)}{[F_{i, j}: \mathbb{Q}]}\Big| \leq c_1\sum_{j\mid \xi_1\#}\sqrt{x}\Big(\log x+\lambda\log j+C(i)\Big).
$$
As $j\leq x^{1/3}$, there exists some $c_2(i)$ such that 
$$
\Big| \sum_{j\mid \xi_1\#}\mu(j)P_{A, i}(x, j) - \sum_{j\mid \xi_1\#}\mu(j)\frac{\li(x)}{[F_{i, j}: \mathbb{Q}]}\Big| \leq c_2(i)\sum_{j\mid \xi_1\#}\sqrt{x}\log x.
$$
Thus by Lemma \ref{O(sqrtlogx)}
\begin{equation}\label{Psumli}
    \sum_{j\mid \xi_1\#}\mu(j)P_{A, i}(x, j) = \sum_{j\mid \xi_1\#}\mu(j)\frac{\li(x)}{[F_{i, j}: \mathbb{Q}]}+O\Big(\frac{x}{\log^2x}\Big).
\end{equation}
where the $O$ term depends on $i$. Finally, combining (\ref{initOform}), (\ref{inexeq}), and (\ref{Psumli}), we see that
\begin{align*}
    N_{A,i}(x) &= \sum_{j\mid \xi_1\#}\mu(j)\Big(\frac{\li(x)}{[F_{i, j}: \mathbb{Q}]}\Big)+O\Big(\frac{x}{\log^2x}\Big) + O\Big(\frac{x\log\log x}{\log^2x}\Big)\\
    &= \Big(\sum_{j\mid \xi_1\#}\frac{\mu(j)}{[F_{i, j}: \mathbb{Q}]}\Big)\li(x) + O\Big(\frac{x\log\log x}{\log^2x}\Big),
\end{align*}
where the $O$ term depends on $i$, as desired.
\end{proof}
\subsection{Results of \cite{https://doi.org/10.48550/arxiv.math/9912250}}
In order to get an exact value for $\delta(a, b)$, Moree and Stevenhagen use the following important fact: 
\begin{lemma}\label{keranh}
Define $\psi: \mathbb{Q}^*/\mathbb{Q}^{*k}\rightarrow \mathbb{Q}(\zeta_k)^*/\mathbb{Q}(\zeta_k)^{*k}$ to be the canonical map. Every element in $\ker\psi$ has order at most 2. 
\end{lemma}
Not knowing the reference for the proof, we shall prove it here:
\begin{proof}
\par Denoting $K=\mathbb{Q}, L=\mathbb{Q}(\zeta_k)$, we note the short exact sequence
\begin{equation}
    \begin{tikzcd}
1 \arrow[r] & \mu_k(L) \arrow[r] & L^* \arrow[r, "x\mapsto x^k"] & L^{*k} \arrow[r] & 1,
\end{tikzcd}
\end{equation}
which induces the long exact sequence of Galois Cohomology
$$
    \begin{tikzcd}
        1 \arrow[r] & {H^0(G_{L/K}, \mu_k(L))} \arrow[r] & {H^0(G_{L/K}, L^*)} \arrow[r] & {H^0(G_{L/K}, L^{*k})} \arrow[r] & {H^1(G_{L/K}, \mu_k(L))} 
    \end{tikzcd}
$$
$$
    \begin{tikzcd}
        \arrow[r] & {H^1(G_{L/K}, L^*)}.
    \end{tikzcd}
$$
By Hilbert Theorem 90, we see that $H^1(G_{L/K}, L^*)=1$. It follows that the following sequence is exact:
\begin{equation*}
    \begin{tikzcd}
        1 \arrow[r] & {\mu_k(K)} \arrow[r] & K^* \arrow[r] & {L^{*k}\cap K^*} \arrow[r] & {H^1(G_{L/K}, \mu_k(L))} \arrow[r] & 1.
    \end{tikzcd}
\end{equation*}
Thus we can see that $L^{*k}\cap K^*/K^{*k}\cong H^1(G_{L/K}, \mu_k(L))$, which means that $\ker\psi\cong H^1(G_{L/K}, \mu_k(L))=H^1((\mathbb{Z}/k\mathbb{Z})^*, \mu_k)$. Now let $\sigma= [x \mapsto \overline{x}]\in G_{L/K}\cong (\mathbb{Z}/k\mathbb{Z})^*$ be the automorphism generated by complex conjugation, $f\in Z^1((\mathbb{Z}/k\mathbb{Z})^*, \mu_{k})$, and $\tau \in (\mathbb{Z}/k\mathbb{Z})^*$. We see that:
\begin{align*}
    f(\tau) &= f(\sigma\tau\sigma) = f(\sigma)f(\tau)^\sigma f(\sigma)^{\sigma\tau} = f(\sigma) f(\sigma)^{\sigma\tau}f(\tau)^\sigma\\
    f(\tau)^2 &= f(\sigma) f(\sigma)^{\sigma\tau}f(\tau)^\sigma f(\tau) = (f(\sigma)/f(\sigma)^\tau)(f(\tau)/f(\tau))\\
    \intertext{as $a^\sigma = 1/a$ for $a\in\mu_{k}$. It follows that}
    f(\tau)^2 &= f(\sigma)/f(\sigma)^\tau
\end{align*}
This means that $f^2\in B^1((\mathbb{Z}/k\mathbb{Z})^*, \mu_{k})$ and so all elements of $H^1((\mathbb{Z}/k\mathbb{Z})^*, \mu_{k})$ have order at most 2. It follows that all elements of $\ker\psi$ have order at most 2. 
\end{proof}
\begin{prop}\textup{\cite[Prop. 4.1]{Wagstaff1982PseudoprimesAA}}\label{wagstaffprop}
    Let $a\in\mathbb{Q}^*_{>0}$ be such that $\widehat{a}=a$. Writing $k'=k/(k, h)$, and $[\mathbb{Q}(\zeta_k, a^{h/k}):\mathbb{Q}]=\varphi(k)k'/\varepsilon(k) $ for some function $\varepsilon(k):\mathbb{N}\rightarrow\mathbb{N}, \varepsilon(k)=2$ if and only if $\Delta(a)\mid k$. 
\end{prop}
From these facts and results from \cite{STEPHENS1976313}, the value $\delta(a, b)$ is found in \cite{https://doi.org/10.48550/arxiv.math/9912250} to be a rational multiple of the Stephens' constant. 

\subsection{Radical Extensions in the $\nu+1$ variable case}\label{radextsec}
\par Let us define $W_{i,j}=\langle a, b_1^j,\ldots, b_{\nu}^j \rangle\subset\mathbb{Q}^*$, which does not depend on $i$. Now denote $ij=k$, with which we may write $W_{i,j}=\langle a^{1/k}, b_1^{1/i},\ldots, b_\nu^{1/i}\rangle^k$ and so $F_{i, j}=\mathbb{Q}(\sqrt[k]{W_{i,j}})$. Here, $W^{1/k}$ is defined as the subgroup in $\overline{\mathbb{Q}}^*$ such that $x\in W^{1/k}$ if and only if $x^{k}\in W$ for some given subgroup $W$ in $\mathbb{Q}^*$ and $k\in\mathbb{Z}$. In particular, $\zeta_k\in W^{1/k}$. We shall also define $\overline{W}_{i,j}$ as the image of $W_{i,j}$ in $\mathbb{Q}^*/\mathbb{Q}^{*k}$. Clearly, $\#\overline{W}_{i,j}\mid i^{\nu+1}j$, which means we can define $t_{i,j}$ such that $\#\overline{W}_{i,j}\cdot t_{i, j}=i^{\nu+1}j$.
\begin{defin}
    Let $x\in\mathbb{Q}^*$. Then $\overline{x}=\widehat{x}\textup{ mod } \mathbb{Q}^{*k}$.
\end{defin}
\begin{defin}
    Let $c_1, \ldots, c_\nu\in\mathbb{Q}^*$ be multiplicatively independent. Then they are \textit{strongly multiplicatively independent} given that they satisfy the condition that if $\prod c_h^{l_h}= c$ then $m_{c_h}l_h= 0 \textup{ mod } m_c$ for all $h$.
\end{defin}
\begin{lemma}\label{strongmulindequiv}
    Let $c_1, \ldots, c_\nu\in \mathbb{Q}^*_{>0}$. Then the following are equivalent:
    \begin{enumerate}
        \item $c_1, \ldots, c_\nu$ are strongly multiplicatively independent
        \item if $W\cong \langle c_1, \ldots, c_\nu\rangle$ then $W\textup{ mod } \mathbb{Q}^{*l} = \mathbb{Z}/\frac{l}{(m_{c_1}, l)}\oplus\ldots\oplus\mathbb{Z}/\frac{l}{(m_{c_\nu}, l)}$, generated by $c_1, \ldots, c_n \textup{ mod } \mathbb{Q}^{*l}$ for all $l\in\mathbb{N}$.
        \item $c_1,\ldots, c_\nu$ is multiplicatively independent with the condition that if $c\in \langle c_1, \ldots, c_\nu\rangle$ then $\widehat{c}\in\langle \widehat{c}_1\ldots, \widehat{c}_\nu\rangle$. 
        \item $\textup{tor }\mathbb{Q}^*/\langle -1, c_1, \ldots, c_\nu \rangle\cong \mathbb{Z}/m_{c_1}\oplus\ldots\oplus \mathbb{Z}/m_{c_\nu}$ generated by $ \widehat{c_1}, \ldots, \widehat{c_\nu}$.
    \end{enumerate}
\end{lemma}
\begin{proof}
    \par Let us assume (1). We know immediately that $W\textup{ mod }\mathbb{Q}^{*l}$ is generated by $c_1, \ldots, c_\nu$ and by the definition of strong multiplicative independence we know that $\prod c_h^{l_h}=1 \textup{ mod } \mathbb{Q}^{*m_c}$ if and only if all $m_{c_h}l_h= 0 \textup{ mod }l$, which implies (2). Now let us assume that $c_1, \ldots, c_n$ is not strongly multiplicatively independent. This means that there exists some $c$ such that $\prod c_h^{l_h}=c$ such that at least one of $l_hm_{c_h}\neq 0 \textup{ mod } m_c$. So there exist $l'_h$ such that $\prod (\overline{c_h}^{m_{c_h}})^{l'_h}=1 \textup{ mod } \mathbb{Q}^{*m_c}$ and at least one $h$ such that $m_{c_h}l'_h\neq 0$. It follows that since $W \textup{ mod } \mathbb{Q}^{*m_c}$ is generated by $c_1, \ldots, c_n$ that $W \textup{ mod } \mathbb{Q}^{*m_c}\not\cong \mathbb{Z}/\frac{l}{(m_{c_1}, l)}\oplus\ldots\oplus\mathbb{Z}/\frac{l}{(m_{c_n}, l)}$. 
    \par Now let us assume $(1)$. For $\prod c_i^{l_i}=c$ it follows that $l_im_{c_i}= 0 \textup{ mod } m_c$. So setting $l'_i= l_im_{c_i}/m_c$ we see that $\prod \widehat{c_i}^{l'_i}= \widehat{c}$, and $\widehat{c}\in\langle \widehat{c_1}, \ldots, \widehat{c_\nu}\rangle$. On the other hand let us assume $(3)$. As $c_1, \ldots, c_\nu$ is multiplicatively independent, it follows that $\widehat{c_1}, \ldots, \widehat{c_\nu}$ is also multiplicatively independent. And so if $\prod c_i^{l_i}= c$ then all the $l_i$'s are unique and likewise if $\prod\widehat{c_i}^{l'_i}= \widehat{c}$ then all the $l'_i$'s are unique. Now for 
    $$
    \prod_{i=1}^\nu \widehat{c_i}^{l_im_{c_i}}= c
    $$
    and
    $$
    \prod_{i=1}^\nu \widehat{c_i}^{l'_im_{c}}= c
    $$
    by uniqueness it follows that $l_im_{c_i}=l'_im_{c}$ for all $i$ and so $l_im_{c_i}= 0 \textup{ mod } m_c$ for all $i$, showing $(1)$. 
    \par Let $\textup{tor }\mathbb{Q}^*/\langle -1, c_1, \ldots, c_\nu\rangle = \textup{tor }\mathbb{Q}^*_{>0}/\langle c_1, \ldots, c_\nu\rangle$ not be generated by $\widehat{c_1}, \ldots, \widehat{c_\nu}$. So there exists some $c\in \textup{tor }\mathbb{Q}^*/\langle -1, c_1, \ldots, c_\nu\rangle$ such that it cannot be written as a product of $\widehat{c_1}, \ldots, \widehat{c_\nu}$ modulo $\langle c_1, \ldots, c_\nu\rangle$. We also see that $\widehat{c}\in\textup{tor }\mathbb{Q}^*/\langle -1, c_1, \ldots, c_\nu\rangle$ and $\widehat{c}$ cannot be written as a product of $\widehat{c_1}, \ldots, \widehat{c_\nu}$ modulo $\langle c_1, \ldots, c_\nu\rangle$. It follows that $\widehat{c}\not\in\langle \widehat{c}_1\ldots, \widehat{c}_\nu\rangle$. However, as there is some $l$ such that $\widehat{c}^l=1 \textup{ mod } \langle c_1, \ldots, c_\nu\rangle$. This therefore means that $\widehat{c}^l=\prod c_i^{l_i}$ for some $l_i\in\mathbb{Z}$, which implies that $\widehat{c}^l\in \langle c_1, \ldots, c_\nu\rangle$ but $\widehat{c}\not\in \langle \widehat{c_1}, \ldots, \widehat{c_\nu}\rangle$. It follows then that $(3)$ implies $(4)$. On the other hand, let us assume $c\in \langle c_1, \ldots, c_\nu\rangle$ but $\widehat{c}\not\in\langle \widehat{c_1}\ldots, \widehat{c_\nu}\rangle$. It is immediately clear that $c\neq \widehat{c}$ and so $\widehat{c}\in \textup{tor }\mathbb{Q}^*/\langle -1, c_1, \ldots, c_\nu\rangle$. However as $\widehat{c}\not\in\langle \widehat{c_1}\ldots, \widehat{c_\nu}\rangle, \widehat{c}\not\in\langle \widehat{c_1}\ldots, \widehat{c_\nu}\rangle\textup{ mod }\langle c_1, \ldots, c_\nu\rangle$, and so $\textup{tor }\mathbb{Q}^*/\langle -1, c_1, \ldots, c_\nu\rangle$ is not generated by $\widehat{c_1}\ldots, \widehat{c_\nu}$. This gives us that $(4)$ implies $(3)$, completing our proof.
\end{proof}
We shall now prove the following:
\begin{lemma}\label{fijlem}
Define $f_{i, j}$ such that 
$$
f_{i,j}=\frac{i^{\nu+1}j\varphi(ij)}{[F_{i,j}:\mathbb{Q}]}.
$$
Then $f_{i,j}=t_{i,j}\cdot \#\ker\psi_{i,j}$ where 
\begin{equation}\label{psiform}
    \psi_{i,j} : \overline{W}_{ i,j}\rightarrow\mathbb{Q}(\zeta_k)^*/\mathbb{Q}(\zeta_k)^{*k}
\end{equation}
is the canonical map.
\end{lemma}
\begin{proof}
By Kummer Theory \cite[Thm. VI.8.1]{lang02}, we know that a subgroup $\Delta\subset \mathbb{Q}(\zeta_k)^*/\mathbb{Q}(\zeta_k)^{*k}$ corresponds bijectively to the extension $\mathbb{Q}(\sqrt[k]{\Delta})$ over $\mathbb{Q}(\zeta_k)$ such that $[\mathbb{Q}(\sqrt[k]{\Delta}):\mathbb{Q}(\zeta_k)]=\#\Delta$. As such, the degree of $F_{i, j}=\mathbb{Q}(\sqrt[k]{W_{i, j}})$ over $\mathbb{Q}(\zeta_k)$ is equal to $\#\im \psi_{i,j}$ by definition. Since $\#\ker\psi_{i,j}\cdot\#\im\psi_{i,j}=\#\overline{W}_{i,j}$ and $t_{i,j}\cdot\#\overline{W}_{i,j}=i^{\nu+1}j$, it follows that
$$
t_{i, j}\cdot\#\ker\psi_{i,j} = \frac{i^{\nu+1}j}{[F_{i, j}:\mathbb{Q}(\zeta_k)]},
$$
but the right hand side is equal to $f_{i, j}$.
\end{proof}
\begin{lemma}\label{tijlem}
$$
t_{i,j}=(ij, m_a)\cdot\prod_{h=1}^\nu (i, m_{b_h})
$$
\end{lemma}
\begin{proof}
$a, b_1, \ldots, b_\nu$ is strongly multiplicatively independent, and so by Lemma \ref{strongmulindequiv} $\overline{W}_{i, j}\cong \mathbb{Z}/\frac{k}{(m_{a}, l)}\oplus\mathbb{Z}/\frac{k}{(m_{b_1}, l)}\oplus\ldots\oplus\mathbb{Z}/\frac{l}{(m_{b_\nu}, l)}$. It follows that $\#\overline{W}_{i,j}$ is $ij/(m_a, ij)\cdot\prod_{h=1}^\nu i/(m_{b_h}, ij)$, which by the identity $\#\overline{W}_{i, j}\cdot t_{i, j}=i^{\nu+1}j$ means that $t_{i, j}=(ij, m_a)\cdot\prod_{h=1}^\nu (i, m_{b_h})$, as desired. 
\end{proof}
\begin{corollary}\label{boundondeg}
$$
\frac{i^{\nu+1}j\varphi(ij)}{[F_{i,j}:\mathbb{Q}]} \mid 2^{\nu+1}m_a\cdot \prod_{h=1}^\nu m_{b_h}
$$
\end{corollary}
\begin{proof}
By Lemma \ref{keranh}, we see that $\ker\psi_{i,j}$ is annihilated by 2. As $\overline{W}_{i,j}$ is generated by $\nu+1$ elements, $\#\ker\psi_{i, j}\mid 2^{\nu+1}$. On the other hand by Lemma \ref{tijlem} we see that 
$$
t_{i, j}\mid m_a\cdot\prod_{h=1}^\nu m_{b_h}
$$
And so by Lemma \ref{fijlem} it follows that
$$
\frac{i^{\nu+1}j\varphi(ij)}{[F_{i,j}:\mathbb{Q}]} \mid 2^{\nu+1}m_a\cdot \prod_{h=1}^\nu m_{b_h}
$$
\end{proof}
While the proof of Corollary \ref{boundondeg} gives the upper bound $2^{\nu+1}$ for the size of $\ker\psi_{i, j}$, in order to control its value we define the following number depending on $x=a^{n_0} \prod b_h^{n_h}$ (see definition of $c_{b_h}$ in Section \ref{notationsection}): 
$$
C_x = \max_{h\in [1, \nu]}\{(c_{b_h}+1)n_h\}.
$$
As such $a, b_1, \ldots, b_\nu$ are multiplicatively independent these $n_h$'s are unique and so $C_x$ is well defined. Using this we prove the following lemma:
\begin{lemma}
    All elements in $\ker\psi_{i, j}$ can be written in the form $x^{ij/2} \textup{ mod } \mathbb{Q}^{*ij}$ where $x\in V$. Then, for $x\in \langle \widehat{b_1}, \ldots, \widehat{b_\nu}\rangle/\langle \widehat{b_1}, \ldots, \widehat{b_\nu}\rangle^2, x^{ij/2}\in\ker\psi_{i,j}$ if and only if $2^{C_x}\mid i$ and $\Delta(\widehat{x})\mid ij$. Finally, for $x\in \widehat{a}\langle \widehat{b_1}, \ldots, \widehat{b_\nu}\rangle/\langle \widehat{b_1}, \ldots, \widehat{b_\nu}\rangle^2, x^{ij/2}\in\ker\psi_{i,j}$ if and only if $2^{C_x}\mid i$ and $2^{c_a+1}, \Delta(\widehat{x})\mid ij$.
\end{lemma}\label{discrem}
\begin{proof}
\par Since $\ker\psi_{i,j}$ is annihilated by 2, it follows that all elements of the kernel are of the form $x^{ij/2} \textup{ mod } \mathbb{Q}^{*ij}$ where $x\in \mathbb{Q}^*$ and $x^{ij/2}$ is generated by $a, b_1^j, \ldots, b_\nu^j$. We see immediately that $2\nmid m_x$ as otherwise it will get annihilated by 2. Furthermore, since $a, b_1, \ldots, b_\nu$ are strongly multiplicatively independent, we see that for $a^{l_0}\prod b_h^{jl_h}=x^{ij/2}$ that $x\in \langle \widehat{a}, \widehat{b_1}, \ldots, \widehat{b_\nu}\rangle$. And finally, since we considering $x^{ij/2}$ modulo $ij$-th powers, we may say without loss of generality that $x\in \langle \widehat{a}, \widehat{b_1}, \ldots, \widehat{b_\nu}\rangle/\langle \widehat{a}, \widehat{b_1}, \ldots, \widehat{b_\nu}\rangle^2=V$. Since $a, b_1, \ldots, b_\nu$ is strongly mulitplicatively independent, so is $\widehat{a}, \widehat{b_1}, \ldots, \widehat{b_\nu}$. So it follows by that for $x=\widehat{a}^{l_0}\prod \widehat{b_h}^{l_h}$ that $m_{\widehat{a}}l_0, m_{\widehat{b_h}}l_h=0\textup{ mod } m_x$ and since $m_{\widehat{a}}l_0, m_{\widehat{b_h}}l_h= 0, 1$, this means that $m_x=1$. And so $\widehat{x}=x$. 
\par Now, for any given element $x\in V, \widehat{x}^{ij/2}\in\ker\psi_{i,j}$ is equivalent to the conditions that $\overline{x}^{ij/2}\in \overline{W}_{i,j}$ and that $[\mathbb{Q}(\zeta_{ij}, \widehat{x}^{1/2}):\mathbb{Q}]=\varphi(ij)$. We shall first show that $\overline{a}^{ij/2}\in \overline{W}_{i,j}$ if and only if $2^{c_a+1}\mid ij$ and that likewise $\overline{x}^{ij/2}\in \overline{W}_{i,j}$ if and only if $2^{C_x}\mid i$ for $x\in \langle b_1, \ldots, b_\nu\rangle/\langle b_1, \ldots, b_\nu\rangle^2$. Finally, for all remaining cases of $x$, we shall show that $\overline{x}^{ij/2}\in \overline{W}_{i,j}$ if and only if $2^{C_x}\mid i$ and $2^{c_a+1}\mid ij$. Having shown this, we show that $[\mathbb{Q}(\zeta_{ij}, \widehat{x}^{1/2}):\mathbb{Q}]=\varphi(ij)$ if and only if $\Delta(\widehat{x})\mid ij$. 
\par Denote $m_x'=m_x/2^{c_x}$, noting $(m_x', 2)=1$. Let us assume $2^{c_a+1}\mid ij$. It follows that $a^{ij/2^{c_a+1}} \textup{ mod } \mathbb{Q}^{*ij} = \widehat{a}^{m_a' ij/2} \textup{ mod } \mathbb{Q}^{*ij} = \widehat{a}^{ij/2} \textup{ mod }\mathbb{Q}^{*ij}$ and so $\widehat{a}^{ij/2}\in\overline{W}_{i,j}$. On the other hand let us assume $2^{c_a+1}\nmid ij$, and let $r_x = m_x \textup{ mod } ij$. As a result of the assumption $2^{c_a+1}\nmid ij$ we see that $c_a\geq \textup{ord}_2(ij)$ and so $\textup{ord}_2(r_a)\geq \textup{ord}_2(ij)$. This means that the order of cyclic group in $\mathbb{Q}^*/\mathbb{Q}^{*ij}$ generated by $a, \#\langle a \mod \mathbb{Q}^{*ij}\rangle = \#\langle \widehat{a}^{r_x}\mod \mathbb{Q}^{*ij}\rangle$ is not divisible by 2 and so $\widehat{a}^{ij/2}\not\in\overline{W}_{i,j}$. We again see that if $2^{c_{b_h}+1}\mid i$ then $\widehat{b_h}^{ij/2}\in\overline{W}_{i,j}$ and if $2^{c_{b_h}+1}\nmid i$ then $\#\langle b_h^{j}\textup{ mod }\mathbb{Q}^{*ij}\rangle=\#\langle \widehat{b_h}^{r_{b_h}j}\textup{ mod }\mathbb{Q}^{*ij}\rangle$ is not divisible by 2, and so $\widehat{b_h}^{ij/2}\not\in\overline{W}_{i,j}$. 
\par Now let $x=\prod \widehat{b_h}^{n_h}\in \langle \widehat{b_1}, \ldots, \widehat{b_\nu}\rangle/\langle \widehat{b_1}, \ldots, \widehat{b_\nu}\rangle^2, n_h=0,1$. If $2^{C_x}\mid i$, we note the following:
\begin{align*}
    \prod_{n_h= 1}(b_h^{j})^{i/2^{b_h+1}} &= \prod_{n_h=1}\widehat{b_h}^{m'_{b_h}ij/2}\\
    &= \prod_{n_h=1}\widehat{b_h}^{ij/2}
\end{align*}
and so $x^{ij/2}\in\overline{W}_{i,j}$. On the other hand, let us say that $x^{ij/2}=\prod b_h^{jl_h}$. By strong multplicative independence this means that $jl_h = \lambda ij $ or $\lambda ij + ij/2$ for some $\lambda\in \mathbb{Z}$. This means that for all $h$ with the condition $jl_h\neq 0 \textup{ mod } ij$ that $2\mid \langle b_h^j \textup{ mod } \mathbb{Q}^{*ij}\rangle$, and so $2\mid \langle b_h \textup{ mod } \mathbb{Q}^{*i}\rangle$. It follows that for such $h$ that $2^{c_{b_h}}\mid i$. And so for $x\in \langle \widehat{b_1}, \ldots, \widehat{b_\nu}\rangle/\langle \widehat{b_1}, \ldots, \widehat{b_\nu}\rangle^2,$ if $ \overline{x}^{ij/2}\in\overline{W}_{i,j}$ then $2^{C_x+1}\mid i$. As a result, we get that $x\in \langle \widehat{b_1}, \ldots, \widehat{b_\nu}\rangle/\langle \widehat{b_1}, \ldots, \widehat{b_\nu}\rangle^2, \overline{x}^{ij/2}\in\overline{W}_{i,j}$ if and only if $2^{C_x+1}\mid i$. We also see that for $x\in \widehat{a}\langle \widehat{b_1}, \ldots, \widehat{b_\nu}\rangle/\langle \widehat{b_1}, \ldots, \widehat{b_\nu}\rangle^2$, applying the same steps as above gives us that $\overline{x}^{ij/2}$ if and only if $2^{C_x}\mid i$ and $2^{c_a+1}\mid ij$. 
\par Finally, by Proposition \ref{wagstaffprop}, we see that for any $x\in V, [\mathbb{Q}(\zeta_{ij}, \widehat{x}^{1/2}):\mathbb{Q}]=\varphi(ij)$ is equivalent to the condition that $\Delta(\widehat{x})\mid ij$, proving our lemma.
\end{proof}

\subsection{Generalised Double Sums}
Let us define $Y_{m, n}^{(\nu)}$ such that
\begin{equation}\label{mnnueq}
    Y_{m , n}^{(\nu)} = \sum_{\substack{i=1\\ m\mid i}}^\infty\sum_{\substack{j=1\\ mn\mid ij}}^\infty\frac{\mu(j)}{i^{\nu+1}j\varphi(ij)},
\end{equation}
and define $S_{m, n}^{(\nu)}$ such that
\begin{equation}\label{Smn}
    S_{m, n}^{(\nu)} = \sum_{\substack{i=1\\ m\mid i}}^\infty\sum_{\substack{j=1\\ n\mid ij}}^\infty\frac{\mu(j)}{i^{\nu+1}j\varphi(ij)} = Y_{m,[m,n]/m}^{(\nu)}
\end{equation}
In order to evaluate $\delta(a, b_1, \ldots, b_{\nu})$, we need to determine the value of this series, for which we can proceed by imitating in \cite[Thm. 4.2]{https://doi.org/10.48550/arxiv.math/9912250}. However, due the nature of these higher powers, we need to define a new function: 
\begin{lemma}
Define $\varphi_n: \mathbb{N}\rightarrow\mathbb{N}$ such that
$$
\varphi_n(x) = x^n\sum_{d\mid x}\frac{\mu(d)}{d^n}.
$$
Then $\varphi_n$ is a multiplicative function.
\end{lemma}
\begin{proof}
By the Euler product formula, we note that
$$
\sum_{d\mid x}\frac{\mu(d)}{d^n} = \prod_{p\mid x}\Big(1-\frac{1}{p^n}\Big),
$$
which means that 
$$
\varphi_n(x) = x^n\prod_{p\mid x}\Big(1-\frac{1}{p^n}\Big).
$$
It follows that $\varphi_n$ is multiplicative.
\end{proof}

\begin{remark}
    On account of the close association with this family of functions with the totient function, a natural question to ask would be if the function $\varphi_n(x)$ has any arithmetic meaning. One meaning is that $\varphi_n(x)$ measures the size of the set $\{k\leq x^n\mid \textup{there exist no primes }p \textup{ such that } p^n \mid (k, x^n)\}$. Likewise, as the totient function, for $n>1$, we can see that the function $\varphi_n(x)$ has average order $\frac{1}{(n+1)\zeta(n+1)}$ using a similar method as detailed in \cite[Ch. 3.7]{apostol}:
\end{remark}
\begin{proof}
    \begin{align*}
        \sum_{k\leq x} \varphi_n(k) &= \sum_{k\leq x}\sum_{d\mid k}\mu(d)\frac{k^n}{d^n}\\
        &= \sum_{q,d, qd\leq x}\mu(d)q^n\\
        &= \sum_{d\leq x}\mu(d)\sum_{q\leq x/d} q^n\\
        \intertext{We use a standard fact about sums that $|\sum_{q\leq w} q^n-\frac{w^{n+1}}{n+1}|\leq L_nw^n$, where $L_n>0$ depends on $n$, to state that}
        &= \sum_{d\leq x}\mu(d)\left(\frac{x^{n+1}}{n+1}\frac{1}{d^{n+1}} + O\left(\frac{x^n}{d^n}\right)\right)\\
        &= \frac{x^{n+1}}{n+1}\sum_{d\leq x}\frac{\mu(d)}{d^{n+1}} + O\left(x^n\sum_{d\leq x}\frac{1}{d^n}\right)\\
        \intertext{Since $O\left(\sum_{d\leq x}\frac{1}{d^n}\right)=O(1)$}
        &= \frac{x^{n+1}}{n+1}\left(\frac{1}{\zeta(n+1)}-\sum_{d > x}\frac{\mu(d)}{d^{n+1}}\right) + O\left(x^n\right)\\
        &= \frac{x^{n+1}}{n+1}\left(\frac{1}{\zeta(n+1)}+O(\frac{1}{x^n})\right) + O\left(x^n\right)\\
        &= \frac{x^{n+1}}{n+1}\frac{1}{\zeta(n+1)} + O\left(x^n\right)
    \end{align*}
    Thus $\varphi_n(x)$ has average order $\frac{1}{(n+1)\zeta(n+1)}$
\end{proof}
With this in hand, we may state the following:
\begin{theorem}\label{Ymnval}
The series (\ref{mnnueq}) is equal to
$$
Y_{m,n}^{(\nu)} = \frac{S^{(\nu)}}{m^{\nu+2}n^{\nu+2}}\prod_{p\mid n}\frac{-p^{\nu+3}(p^\nu-1)}{p^{\nu+3}-p^{\nu+2}-p^{\nu+1}+1}\prod_{\substack{p\mid m\\ p\nmid n}}\frac{p^{\nu+1}(p^2-1)}{p^{\nu+3}-p^{\nu+2}-p^{\nu+1}+1},
$$
where 
$$
S^{(\nu)} = \prod_p\Big(1-\frac{p^\nu-1}{p-1}\Big(\frac{p}{p^{\nu+2}-1}\Big)\Big).
$$
\end{theorem}
\begin{proof}
Substituting $mi$ for $i$ 
\begin{align*}
    Y_{m,n}^{(\nu)} &= \frac{1}{m^{\nu+1}}\sum_{\substack{i=1\\ m\mid i}}^\infty\sum_{\substack{j=1\\ mn\mid ij}}^\infty\frac{\mu(j)}{i^{\nu+1}j\varphi(mij)}= \frac{1}{m^{\nu+1}}\sum_{\substack{i=1\\ m\mid i}}^\infty\sum_{\substack{j=1\\ mn\mid ij}}^\infty\frac{\mu(j)}{(ij)^{\nu+1}\varphi(mij)}j^{\nu}\\
    \intertext{and $nd$ for $ij$}
    &= \frac{1}{m^{\nu+1}n^{\nu+1}}\sum_{d=1}^\infty\frac{1}{d^{\nu+1}\varphi(mnd)}\sum_{j\mid nd}j^\nu\mu(j).
\end{align*}
Denoting $\widetilde{x}=\prod_{p\mid x}p$, we see that
$$
\sum_{j\mid x}j^\nu\mu(j)=\sum_{j\mid \widetilde{x}}j^\nu\mu(j)=\mu(\widetilde{x})\sum_{j\mid \widetilde{x}}j^\nu\mu(\widetilde{x}/j)=\mu(\widetilde{x})\varphi_\nu(\widetilde{x})
$$
This gives us that
$$
Y_{m,n}^{(\nu)} = \frac{\mu(\widetilde{n})\varphi_\nu(\widetilde{n})}{m^{\nu+1}n^{\nu+1}\varphi(mn)}\sum_{d=1}^\infty f(d)
$$
Where
$$
f(d)=\frac{\varphi(mn)\mu(\widetilde{nd})\varphi_\nu(\widetilde{nd})}{d^{\nu+1}\varphi(mnd)\mu(\widetilde{n})\varphi_\nu(\widetilde{n})} = \frac{\mu(\widetilde{nd})}{d^{\nu+1}\mu(\widetilde{n})}\prod_{\substack{p\mid d\\ p\mid mn}}p^{-\textup{ord}_p(d)}\prod_{\substack{p\mid d\\ p\nmid mn}}\left(p^{\textup{ord}_p(d)}-p^{\textup{ord}_p(d)-1}\right)^{-1}\prod_{\substack{p\mid d\\ p\nmid n}}\left(p^\nu-1\right)
$$
is a multiplicative function, as the functions which constitute it are multiplicative, with $\sum_{d=1}^\infty f(d)$ converging absolutely as $|f(d)|\leq \frac{1}{d^{2}}$. This means that we may consider the Euler product expansion of this sum. Now, we see by direct substitution that
$$
f(p^k) = 
\begin{cases}
    -\frac{p-1}{p^\nu-1}p^{1-k(\nu+2)} & p\nmid mn\\
    \frac{1}{p^{k(\nu+2)}} & p\mid n\\
    -\frac{p^\nu-1}{p^{k(\nu+2)}} & p\mid m, p\nmid n
\end{cases}
$$
which gives us that
$$
\sum_{k=0}^{\infty}f(p^k) = 
\begin{cases}
    1- \frac{p^\nu-1}{p-1}\Big(\frac{p}{p^{\nu+2}-1}\Big) & p\nmid mn\\
    \frac{p^{\nu+2}}{p^{\nu+2}-1} & p\mid n\\
    \frac{p^{\nu+2}-p^{\nu}}{p^{\nu+2}-1} & p\mid m, p\nmid n
\end{cases}
$$
\begin{align*}
    Y_{m, n}^{(\nu)} &= \frac{\mu(\widetilde{n})\varphi_\nu(\widetilde{n})}{m^{\nu+1}n^{\nu+1}\varphi(mn)}\prod_{p\mid n}\frac{p^{\nu+2}}{p^{\nu+2}-1}\prod_{\substack{p\mid m\\ p\nmid n}}\frac{p^{\nu+2}-p^{\nu}}{p^{\nu+2}-1}\prod_{p\nmid mn}\Big(1- \frac{p^\nu-1}{p-1}\Big(\frac{p}{p^{\nu+2}-1}\Big)\Big)\\
    \intertext{Since $p\mid \widetilde{n}\Longleftrightarrow p\mid n$}
    &= \frac{1}{m^{\nu+1}n^{\nu+1}\varphi(mn)}\prod_{p\mid n}(-1)\frac{p^\nu-1}{p}\prod_{p\mid n}\frac{p^{\nu+2}}{p^{\nu+2}-1}\prod_{\substack{p\mid m\\ p\nmid n}}\frac{p^{\nu+2}-p^{\nu}}{p^{\nu+2}-1}\prod_{p\nmid mn}\Big(1- \frac{p^\nu-1}{p-1}\Big(\frac{p}{p^{\nu+2}-1}\Big)\Big)\\
    &= \frac{S^{(\nu)}}{m^{\nu+2}n^{\nu+2}}\frac{mn}{\varphi(mn)}\prod_{p\mid n}\frac{p^{\nu+2}(1-p^\nu)(p-1)}{p^{\nu+3}-p^{\nu+2}-p^{\nu+1}+1}\prod_{\substack{p\mid m\\ p\nmid n}}\frac{(p^{\nu+2}-p^{\nu})(p-1)}{p^{\nu+3}-p^{\nu+2}-p^{\nu+1}+1}\\
    &= \frac{S^{(\nu)}}{m^{\nu+2}n^{\nu+2}}\prod_{p\mid mn}\frac{1}{1-\frac{1}{p}}\prod_{p\mid n}\frac{p^{\nu+2}(1-p^\nu)(p-1)}{p^{\nu+3}-p^{\nu+2}-p^{\nu+1}+1}\prod_{\substack{p\mid m\\ p\nmid n}}\frac{(p^{\nu+2}-p^{\nu})(p-1)}{p^{\nu+3}-p^{\nu+2}-p^{\nu+1}+1}\\
    &= \frac{S^{(\nu)}}{m^{\nu+2}n^{\nu+2}}\prod_{p\mid n}\frac{-p^{\nu+3}(p^\nu-1)}{p^{\nu+3}-p^{\nu+2}-p^{\nu+1}+1}\prod_{\substack{p\mid m\\ p\nmid n}}\frac{p^{\nu+1}(p^2-1)}{p^{\nu+3}-p^{\nu+2}-p^{\nu+1}+1}
\end{align*}
\end{proof}
\begin{corollary}\label{Smnval}
The series (\ref{Smn}) is equal to 
$$
S_{m,n}^{(\nu)} = \frac{S^{(\nu)}}{[m,n]^{\nu+2}}\prod_{p\mid \frac{n}{(m,n)}}\frac{-p^{\nu+3}(p^\nu-1)}{p^{\nu+3}-p^{\nu+2}-p^{\nu+1}+1}\prod_{p\mid \frac{m}{(m,n)}}\frac{p^{\nu+1}(p^2-1)}{p^{\nu+3}-p^{\nu+2}-p^{\nu+1}+1}.
$$
\end{corollary}
\begin{proof}
    This comes from using Theorem \ref{Ymnval} and noting that $\frac{[m,n]}{m}=\frac{n}{(m,n)}$ and $p\mid \frac{m}{(m,n)}$ if and only if $p\nmid \frac{n}{(m,n)}$. 
\end{proof}
In addition, we shall also note the following:
\begin{lemma}\label{kerlemma}
Let $\psi_{i,j}$ be defined as in (\ref{psiform}). Define $\{C_{i, j}\}_{i, j}$ to be a sequence such that $\sum_i\sum_j C_{i,j}$ absolutely converges. Then
$$
\sum_{i=1}^\infty\sum_{j=1}^\infty\#\ker\psi_{i, j}\cdot C_{i, j} = \sum_{i=1}^\infty\sum_{j=1}^\infty C_{i, j} + \sum_{x\in \widehat{a}V}\sum_{\substack{i=1\\ 2^{C_x}\mid i}}^\infty\sum_{\substack{j=1\\ 2^{c_a+1}\mid ij\\ \Delta(\widehat{x})\mid ij}}^\infty C_{i, j} + \sum_{\substack{x\in V/\langle\widehat{a}\rangle\\ x\neq 1}}\sum_{\substack{i=1\\ 2^{C_x}\mid i}}^\infty\sum_{\substack{j=1\\ \Delta(\widehat{x})\mid ij}}^\infty C_{i, j}
$$
where $V=\langle \widehat{a}, \widehat{b_1}, \ldots, \widehat{b_\nu} \rangle/\langle \widehat{a}, \widehat{b_1}, \ldots, \widehat{b_\nu} \rangle^2$
\end{lemma}
\begin{proof}
Denoting $G=\langle \overline{a}^{ij/2}, \overline{b}_1^{ij/2}, \ldots,\overline{b}_\nu^{ij/2} \rangle$, we see that
\begin{align*}
    \sum_{i=1}^\infty\sum_{j=1}^\infty\#\ker\psi_{i, j}\cdot C_{i, j} &= \sum_{H\leqslant G }\sum_{\substack{i=1}}^\infty\sum_{\substack{j=1\\ \ker\psi_{i, j} = H}}^\infty \# H\cdot C_{i, j}
    \intertext{Reordering by containment of each element in $G$}
    &= \sum_{i=1}^\infty\sum_{\substack{j=1\\ 1\in \ker\psi_{i, j}}}^\infty C_{i, j} + \sum_{x\in \overline{a}^{ij/2}G}\sum_{i=1}^\infty\sum_{\substack{j=1\\ x\in\ker\psi_{i, j}}}^\infty C_{i, j} + \sum_{\substack{x\in G/\langle\overline{a}^{ij/2}\rangle\\ x\neq 1}}\sum_{\substack{i=1}}^\infty\sum_{\substack{j=1\\ x\in\ker\psi_{i,j}}}^\infty  C_{i, j}
    \intertext{Using the Lemma \ref{discrem}}
    \sum_{i=1}^\infty\sum_{j=1}^\infty\#\ker\psi_{i, j}\cdot C_{i, j} &= \sum_{i=1}^\infty\sum_{j=1}^\infty C_{i, j} + \sum_{x\in \widehat{a}V}\sum_{\substack{i=1\\ 2^{C_x}\mid i}}^\infty\sum_{\substack{j=1\\ 2^{c_a+1}\mid ij \\ \Delta(\widehat{x})\mid ij}}^\infty C_{i, j} + \sum_{\substack{x\in V/\langle\widehat{a}\rangle\\ x\neq 1}}\sum_{\substack{i=1\\ 2^{C_x}\mid i}}^\infty\sum_{\substack{j=1\\ \Delta(\widehat{x})\mid ij}}^\infty C_{i, j}
\end{align*}
\end{proof}
\begin{lemma}\label{switchform1var}
Let $\{C_i\}_{i\in\mathbb{N}}$ be a sequence such that $\sum C_i$ converges absolutely. Then
$$
\sum_{d\mid n}d\sum_{\substack{i=1\\ (n, i)=d}}^\infty C_i = \sum_{d\mid n}\varphi(d)\sum_{\substack{i=1\\ d\mid i}}^\infty C_i
$$
\end{lemma}
\begin{proof}
We note that $d\mid i$ if and only if $d\mid (n, i)$. This means that rearranging terms on the right hand side based on the values of $(n, i)$, we get that 
$$
\sum_{d\mid n}\varphi(d)\sum_{\substack{i=1\\ d\mid i}}^\infty C_i = \sum_{d\mid n}\Big(\sum_{j\mid d}\varphi(j)\Big)\sum_{\substack{i=1\\ (n, i)=d}}^\infty C_i
$$
By a theorem of Gauss we get that $\sum_{j\mid d}\varphi(j)= d$, which means that
$$
\sum_{d\mid n}\varphi(d)\sum_{\substack{i=1\\ d\mid i}}^\infty C_i = \sum_{d\mid n}d\sum_{\substack{i=1\\ (n, i)=d}}^\infty C_i
$$
as desired.
\end{proof}
The following lemma is also proved the same way:
\begin{lemma}\label{switchform2var}
Let $\{C_{i, j}\}_{i, j\in\mathbb{N}}$ be a sequence such that $\sum_i\sum_j C_{i, j}$ converges absolutely. Then
$$
\sum_{d\mid n}d\sum_{i=1}^\infty\sum_{\substack{j=1\\ (n, ij)=d}}^\infty C_i = \sum_{d\mid n}\varphi(d)\sum_{i=1}^\infty\sum_{\substack{j=1\\ d\mid ij}}^\infty C_i
$$
\end{lemma}

\subsection{Evaluation of the Density}
\begin{proof}[Proof of Theorem \ref{bigeqden}]

By Theorem \ref{maindenthm}
\begin{align*}
    \delta(a, b_1,\ldots, b_\nu) &= \sum_{i=1}^\infty\sum_{j=1}^\infty \frac{\mu(j)}{[F_{i, j}:\mathbb{Q}]}.\\
    \intertext{By definition of $f_{i,j}$}
    \delta(a, b_1,\ldots, b_\nu) &= \sum_{i=1}^\infty\sum_{j=1}^\infty \frac{f_{i, j}\mu(j)}{i^{\nu+1}j}.\\
    \intertext{By Lemma \ref{fijlem}}
    \delta(a, b_1,\ldots, b_\nu) &= \sum_{i=1}^\infty\sum_{j=1}^\infty t_{i, j}\cdot\#\ker\psi_{i,j}\cdot\frac{\mu(j)}{i^{\nu+1}j}.\\
    \intertext{Applying \ref{tijlem} and reorganising by size of $t_{i,j}$}
    \delta(a, b_1,\ldots, b_\nu) &= \sum_{\substack{d\mid m_a\\ d_1\mid m_{b_1}\\ \cdots\\ d_\nu\mid m_{b_\nu}}}\Big(d\prod_{h=1}^\nu d_h\Big)\sum_{\substack{i=1\\ (i, m_{b_h})=d_h, \\ \forall h}}^\infty\sum_{\substack{j=1\\ (ij, m_a)=d}}^\infty\#\ker\psi_{i,j}\cdot\frac{\mu(j)}{i^{\nu+1}j}.
\end{align*}
Applying Lemmas \ref{switchform1var} and \ref{switchform2var}, we get that
\begin{equation}
    \delta(a, b_1,\ldots, b_\nu) = \sum_{\substack{d\mid m_a\\ d_1\mid m_{b_1}\\ \cdots\\ d_\nu\mid m_{b_\nu}}}\Big(\varphi(d)\prod_{h=1}^\nu \varphi(d_h)\Big)\sum_{\substack{i=1\\ d_h\mid i, \\ \forall h}}^\infty\sum_{\substack{j=1\\ d\mid ij}}^\infty\#\ker\psi_{i,j}\cdot\frac{\mu(j)}{i^{\nu+1}j}
\end{equation}
And finally, applying Lemma \ref{kerlemma},
\begin{multline*}
    \delta(a, b_1,\ldots, b_\nu) = \sum_{\substack{d\mid m_a\\ d_1\mid m_{b_1}\\ \cdots\\ d_\nu\mid m_{b_\nu}}}\Big(\varphi(d)\prod_{h=1}^\nu \varphi(d_h)\Big)\Big(\sum_{\substack{i=1 \\ d_h\mid i, \\ \forall h}}^\infty\sum_{\substack{j=1\\ d\mid ij}}^\infty\frac{\mu(j)}{i^{\nu+1}j} + \sum_{x\in \widehat{a}V}\sum_{\substack{i=1\\ 2^{C_x+1}, d_h\mid i\\ \forall h}}^\infty\sum_{\substack{j=1\\ 2^{c_a+1},d, \Delta(\widehat{x})\mid ij}}^\infty \frac{\mu(j)}{i^{\nu+1}j} + \\ 
    \sum_{\substack{x\in V/\langle\widehat{a}\rangle\\ x\neq 1}}\sum_{\substack{i=1\\ 2^{C_x+1}, d_h\mid i\\ \forall h}}^\infty\sum_{\substack{j=1\\ d, \Delta(\widehat{x})\mid ij}}^\infty \frac{\mu(j)}{i^{\nu+1}j}\Big)
\end{multline*}
By the fact that $a, b\mid c$ if and only if $[a, b]\mid c$ this means that 
$$
\delta(a, b_1,\ldots, b_\nu) = \sum_{\substack{d\mid m_a\\ d_1\mid m_{b_1}\\ \cdots\\ d_\nu\mid m_{b_\nu}}}\Big(\varphi(d)\prod_{h=1}^\nu \varphi(d_h)\Big)\left(S^{(\nu)}_{[d_h]_h, d}+\sum_{x\in \widehat{a}V}S^{(\nu)}_{[2^{C_x},d_h]_h, [2^{c_a+1}, d, \Delta(\widehat{x})]}+\sum_{\substack{x\in V/\langle\widehat{a}\rangle\\ x\neq 1}}S^{(\nu)}_{[2^{C_x}, d_h]_h, [ d, \Delta(\widehat{x})]} \right)
$$
\end{proof}
\begin{corollary}
When $m_x=1$ for all $x\in V$,
$$
    \delta(a, b_1,\ldots, b_\nu) = S^{(\nu)}_{1, 1} +S^{(\nu)}_{1, [2, \Delta(a)]}+\sum_{\substack{x\in V\\ x\neq 1, a}}S^{(\nu)}_{2, \Delta(x)}.
$$
\end{corollary}
\begin{corollary}\cite[Thm. 3]{https://doi.org/10.48550/arxiv.math/9912250}
When $m_x=1$ for all $x\in V$ and $\nu =2$,
$$
    \delta(a, b) = S^{(2)}_{1, 1} + S^{(\nu)}_{1, [2, \Delta(a)]} + S^{(\nu)}_{2, \Delta(b)} + S^{(\nu)}_{2, \Delta(ab)}
$$ 
\end{corollary}
\begin{proof}[Proof of Theorem \ref{O(disc)nvar}]
    Let us fix some $x\in V, x\neq 1$ and $A, B\in \mathbb{N}$. By Corollary \ref{Smnval}
    \begin{align*}
        \left|S^{(\nu)}_{A, [B, \Delta(\widehat{x})]}\right| &\leq \frac{S^{(\nu)}}{[A, B, \Delta(\widehat{x})]^{\nu+2}}\prod_{p\mid [B, \Delta(\widehat{x})]}\frac{p^{\nu+3}(p^\nu-1)}{p^{\nu+3}-p^{\nu+2}-p^{\nu+1}+1}\prod_{p\mid A}\frac{p^{\nu+1}(p^2-1)}{p^{\nu+3}-p^{\nu+2}-p^{\nu+1}+1}\\
        &\leq \frac{S^{(\nu)}}{[B, \Delta(\widehat{x})]^{\nu+2}}\prod_{p\mid [B, \Delta(\widehat{x})]}\frac{p^{\nu+3}p^\nu}{p^{\nu+3}-p^{\nu+2}-p^{\nu+1}+1}\prod_{p\mid A}\frac{p^{\nu+1}(p^2-1)}{p^{\nu+3}-p^{\nu+2}-p^{\nu+1}+1}\\
        &\leq \frac{S^{(\nu)}[B, \Delta(\widehat{x})]^\nu}{[B, \Delta(\widehat{x})]^{\nu+2}}\prod_{p\mid [B, \Delta(\widehat{x})]}\frac{p^{\nu+3}}{p^{\nu+3}-p^{\nu+2}-p^{\nu+1}+1}\prod_{p\mid A}\frac{p^{\nu+1}(p^2-1)}{p^{\nu+3}-p^{\nu+2}-p^{\nu+1}+1}\\
        &\leq \frac{S^{(\nu)}}{[B, \Delta(\widehat{x})]^{2}}\prod_{p\mid [B, \Delta(\widehat{x})]}\frac{1}{1-\frac{1}{p}-\frac{1}{p^2}}\prod_{p\mid A}\frac{p^{\nu+1}(p^2-1)}{p^{\nu+3}-p^{\nu+2}-p^{\nu+1}+1}.
    \end{align*}
We now show that
$$
\prod_{p\mid n}\frac{1}{1-\frac{1}{p}-\frac{1}{p^2}} = O(\log n).
$$
We first see that 
$$
\prod_{p\mid n}\frac{\left(1-\frac{1}{p}\right)}{\left(1-\frac{1}{p}-\frac{1}{p^2}\right)} = \prod_{p\mid n}\left(1+\frac{1}{p^2-p-1}\right) < k\prod_{p}\left(1+\frac{1}{p^{1.5}}\right)
$$
for some $k>0$. Moreover,
$$
\prod_{p\mid n}\frac{\left(1-\frac{1}{p}\right)}{\left(1-\frac{1}{p}-\frac{1}{p^2}\right)} < k\prod_{p}\left(1+\frac{1}{p^{1.5}}\right) = k\prod_{p}\frac{\left(1-\frac{1}{p^3}\right)}{\left(1-\frac{1}{p^{1.5}}\right)} = k\frac{\zeta(1.5)}{\zeta(3)}
$$
And so 
$$
\prod_{p\mid n}\frac{1}{1-\frac{1}{p}-\frac{1}{p^2}} \leq k\frac{\zeta(1.5)}{\zeta(3)}\prod_{p\leq n}\frac{1}{1-\frac{1}{p}}.
$$
It follows from Mertens' third theorem\cite{weisstein} that
$$
\prod_{p\mid n}\frac{1}{1-\frac{1}{p}-\frac{1}{p^2}} = O(\log n).
$$
With this we may conclude that
$$
S^{(\nu)}_{A, [B, \Delta(\widehat{x})]} = O\left(\frac{\log\Delta(\widehat{x})}{\Delta(\widehat{x})^2}\right),
$$
and so
$$
\delta(a, b_1,\ldots, b_\nu) = \sum_{\substack{d\mid m_a\\ d_1\mid m_{b_1}\\ \cdots\\ d_\nu\mid m_{b_\nu}}}\Big(\varphi(d)\prod_{h=1}^\nu \varphi(d_h)\Big)S^{(\nu)}_{[d_h]_h, d}+\sum_{\substack{x\in V\\ x\neq 1}} O\left(\frac{\log\Delta(\Hat{x})}{\Delta(\widehat{x})^2}\right).
$$
\end{proof}
\begin{corollary}\label{infsetup}
Let $m_a, m_{b_1}, \ldots, m_{b_9}\in\mathbb{N}$ and define a sequence $\{a_i, b_{1, i}, \ldots, b_{\nu, i}\}_{i\in\mathbb{N}}$. such that the torsion group of $\mathbb{Q}^*/\langle -1, a_i, b_{1, i}, \ldots, b_{\nu, i} \rangle$ is $\mathbb{Z}/m_a\oplus\mathbb{Z}/m_{b_1}\oplus\ldots\oplus\mathbb{Z}/m_{b_\nu}$ and $\Delta(\widehat{a_i}), \Delta(\widehat{b_
{1, i}}), \ldots, \Delta(\widehat{b_{\nu, i}})\rightarrow\infty$ as $i\rightarrow\infty$. Then
$$
\lim_{i\rightarrow\infty} \delta(a, b_{1, i},\ldots, b_{\nu,i}) = \sum_{\substack{d\mid m_a\\ d_1\mid m_{b_1}\\ \cdots\\ d_\nu\mid m_{b_\nu}}}\Big(\varphi(d)\prod_{h=1}^\nu \varphi(d_h)\Big)S^{(\nu)}_{[d_h]_h, d}
$$
\end{corollary}
As a result of this corollary, we remove all the "noise" brought about by the discriminant terms, which gives us the approximate values around which our density values hover. This yields the values of the following table:
\begin{table}[h]
    \centering
    \begin{tabular}{|c|c|}
        \hline
        $(m_a, m_{b_1}, \ldots, m_{b_\nu})$ &  $\lim_{i\rightarrow\infty} \delta(a, b_{1, i},\ldots, b_{\nu,i})/S^{(\nu)}$\\
        \hline
        $(1, 1)$ & 1\\
        \hline
        $(2, 1)$ & 3/5\\
        \hline
        $(2, 2)$ & 41/40 \\
        \hline
    \end{tabular}
    \caption{Some values of $\lim_{i\rightarrow\infty} \delta(a, b_{1, i},\ldots, b_{\nu,i})$}
    \label{tab:my_label}
\end{table}

\section{Torsion of Root Lattices}
\par Let $(C, X)$ be an elliptic pair defined as in Section \ref{polyhedsec}. Namely, with $\Gamma\subset\mathbb{P}^2$ being a nodal cubic defined over $\mathbb{Q}$ with $[0:1:0]$ as a flex point, let $X$ be the blow-up of $\mathbb{P}^2$ at nine distinct points $z_1, \ldots, z_9\in \Gamma$ and $C$ the proper transform of $\Gamma$ under this blow-up. With this, we define $X_p$ as the geometric fiber of $X$ over the prime $p$. Also in Section \ref{polyhedsec} we had established that the problem of calculating the density of primes $p$ where $\overline{\textup{Eff}}(X_p)$ is polyhedral, denoted $\delta(X)$, can be reduced to the problem of calculating the density of primes such that there exists a rank 8 root lattice $\Lambda\subset \mathbb{E}_8$ with the property $\Lambda\subseteq \ker\overline{\textup{res}}_p$. As the smooth locus $\Gamma$ can be identified with $\mathbb{Q}^*$, $z_1, \ldots, z_9$ can likewise be identified with elements in $\mathbb{Q}^*$, which by abuse of notation we shall also denote $z_1, \ldots, z_9$, noting that we assume $z_1, \ldots, z_9$ multiplicatively independent. From here we see that the restriction map $\textup{res}: \textup{Pic } X \rightarrow \textup{Pic } C$ induces the map $\textup{res}: C^{\perp} \rightarrow \textup{Pic}^0 C$. We see that the canonical class $K=-C\in C^{\perp}$ and so we get the map $\overline{\textup{res}}: C^{\perp}/\langle C\rangle \rightarrow \textup{Pic}^0 C/\langle\textup{res } C\rangle$. As $C^{\perp}/\langle C\rangle\cong \mathbb{E}_8, \textup{Pic}^0 (C)(\mathbb{Q})\cong\mathbb{Q}^*$, and $\textup{res}(C)=z_1\ldots z_9$ we get 
$$
\overline{\textup{res}}: \mathbb{E}_8 \rightarrow \mathbb{Q}^*/\langle z_1\ldots z_9\rangle.
$$
This allows for the following identification of $\mathbb{E}_8$:
\begin{figure}[!ht]
    \centering
    \begin{tikzpicture}
    \draw[black, thick] (-3, 1) -- (4, 1);
    \draw[black, thick] (-1, 1) -- (-1, 0);
    \filldraw[black] (-1, 0) circle (4pt);
    \filldraw[black] (-1, 1) circle (4pt);
    \filldraw[black] (-2, 1) circle (4pt);
    \filldraw[black] (-3, 1) circle (4pt);
    \filldraw[black] (0, 1) circle (4pt);
    \filldraw[black] (1, 1) circle (4pt);
    \filldraw[black] (2, 1) circle (4pt);
    \filldraw[black] (3, 1) circle (4pt);
    \draw[black, thick] (4,1) node {$\bigtimes$};
    \draw[black, thick] (4,1.5) node {$z_8/z_9$};
    \draw[black, thick] (3,1.5) node {$z_7/z_8$};
    \draw[black, thick] (2,1.5) node {$z_6/z_7$};
    \draw[black, thick] (1,1.5) node {$z_5/z_6$};
    \draw[black, thick] (0,1.5) node {$z_4/z_5$};
    \draw[black, thick] (-1,1.5) node {$z_3/z_4$};
    \draw[black, thick] (-2,1.5) node {$z_2/z_3$};
    \draw[black, thick] (-3,1.5) node {$z_1/z_2$};
    \draw[black, thick] (-1,-.5) node {$1/z_1z_2z_3$};
\end{tikzpicture}
    \caption{Embedding of $\mathbb{E}_8$ into $\mathbb{Q}^*/\langle z_1\ldots z_9\rangle$}
\end{figure}
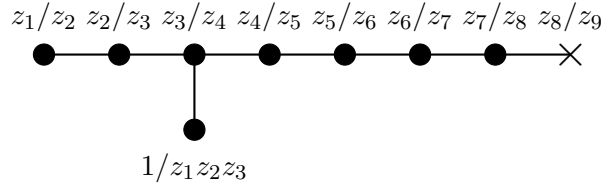
\par This identification is an injective embedding: to show $\langle z_1/z_2, \ldots, z_7/z_8, 1/z_1z_2z_3\rangle$ is multiplicatively independent, it suffices to show that it is linearly independent in $\langle z_1, \ldots, z_9\rangle/\langle z_1\ldots z_9\rangle = \langle z_1, \ldots, z_8\rangle$, and since the simple roots $\mathbb{E}_8$ form the basis of a rank 8 submodule in $\mathbb{Z}^8$, it follows that $\langle z_1/z_2, z_2/z_3, \ldots, z_7/z_8, 1/z_1z_2z_3\rangle$ is multiplicatively independent. Furthermore for any prime $p$, we may define $(C_p, X_p)$ as the geometric fiber of $(C, X)$ over $p$. The restriction map in this case $\textup{res}_p:\textup{Pic }X_p\rightarrow\textup{Pic }C_p$ similarly induces the map $\overline{\textup{res}}_p:C_p^{\perp}/\langle C_p\rangle\rightarrow \textup{Pic}^0C_p/\langle\textup{res}_p(C)\rangle$, which since $\textup{Pic}^0 (C)(\mathbb{F}_p)\cong\mathbb{F}_p^*$, gives us the map:
$$
\overline{\textup{res}}_p:\mathbb{E}_8\rightarrow \mathbb{F}^*_p/\langle z_1\ldots z_9 \textup{ mod } p \rangle.
$$
By \cite[Thm. 3.8]{https://doi.org/10.48550/arxiv.2009.14298} we see that $\overline{\textup{Eff}}(X_p)$ is polyhedral if and only if there is a rank 8 root lattice $\Lambda\subseteq\mathbb{E}_8$ such that $\Lambda\subseteq\ker\overline{\textup{res}_p}$, which implies that $\overline{\textup{Eff}}(X_p)$ is polyhedral if and only if 
there exist 8 elements $b_1, \ldots, b_8 \in \langle z_1/z_2, z_2/z_3, \ldots, z_8/z_9, 1/z_1z_2z_3\rangle$ which generates a root lattice in $\mathbb{E}_8$ of rank 8 such that $b_1, \ldots, b_8\in \langle z_1\ldots z_9\rangle \textup{ mod } p$. And so the problem of calculating $\delta(X)$ reduces to finding the density of primes $p$ such that there exist 8 elements $b_1, \ldots, b_8 \in \langle z_1/z_2, z_2/z_3, \ldots, z_8/z_9, 1/z_1z_2z_3\rangle$ which generate a root lattice in $\mathbb{E}_8$ of rank 8 such that $b_1, \ldots, b_8\in \langle z_1\ldots z_9\rangle \textup{ mod } p$. In Section 3, we have shown under the assumption of the Generalized Riemann Hypothesis that for multiplicatively independent $a, b_1, \ldots, b_8$ the density that $b_1, \ldots, b_8\in \langle a \rangle \textup{ mod }p$ is a rational multiple of $S^{(8)}$, proving Theorem \ref{ratmulnvar}. While the statement of Theorem \ref{ratmulnvar} suffices to prove Theorem \ref{mainresult}, it fails to produce a specific value. To do this let us define $\overline{\delta}(\Lambda)$ to be the density of primes $p$ such that $\ker\overline{\textup{res}_p} = \Lambda$ for some rank 8 root lattice $\Lambda$. We notice immediately that $\overline{\delta}(\mathbb{E}_8)=\delta(\mathbb{E}_8)$, and in the case of other rank 8 lattices $\Lambda$ that $\overline{\delta}(\Lambda)$ equals the density of primes such that $\langle z_1\ldots z_9\rangle \textup{ mod }p \equiv \Lambda\textup{ mod }p$. This reduces the problem of getting an explicit value of $\delta(X)$ to calculating $\overline{\delta}(\Lambda)$ for each rank 8 root lattice $\Lambda$, which requires knowing the number of lattice of each type and $\textup{tor }\mathbb{Q}^*/\langle z_1\ldots z_9, \Lambda\rangle$ for any given root lattice $\Lambda$. The task of counting the root lattices can be determined via \cite[Eqn. 10.3]{oshima2007classification}:
\begin{equation*}
    \#O_\Lambda = (\#)\cdot\#W_{\mathbb{E}_8}/((\#_{\mathbb{E}_8})\cdot\#\textup{Out}(\Lambda)\cdot\# W_{\Lambda}\cdot\# W_{\Lambda^{\perp}})
\end{equation*}
Here $O_\Lambda$ is the set of all root lattices of type $\Lambda$, $(\#) ,\#_\Lambda$ are given by \cite[Table 10.2]{oshima2007classification}, $W_\Lambda$ denotes the Weyl group of $\Lambda$, $N_\Lambda$ denotes the Normaliser of $\Lambda$ in $W_{\mathbb{E}_8}$, and $\textup{Out}(\Lambda)$ defined as in \cite[Sec. 2]{oshima2007classification}. As $\Lambda$ is of rank 8, $\Lambda^{\perp}=\O$ and so $\# W_{\Lambda^{\perp}}=1$. Applying the formula, we get the following:
\begin{table}[h]
    \centering
    \begin{tabular}{|c|c|c|c|c|c|c|}
\hline
{ $\Lambda$}        & { $\#\textup{Out}(\Lambda)$} & { $(\#)$} & { $\#_\Lambda$} & { $\#W_\Lambda$}      & { $\#N_\Lambda$}    & { $\#O_\Lambda$}  \\
\hline
{ $\mathbb{A}_8$}       & { 2}        & { 1}  & { 1}   & { 362880}    & { 725760}    & { 960}   \\
\hline
{ $\mathbb{D}_8$}       & { 2}        & { 2}  & { 1}   & { 5160960}   & { 5160960}   & { 135}   \\
\hline
{ $\mathbb{E}_7\oplus \mathbb{A}_1$}    & { 1}        & { 1}  & { 1}   & { 5806080}   & { 5806080}   & { 120}   \\
\hline
{ $\mathbb{A}_5+\mathbb{A}_2+\mathbb{A}_1$} & { 4}        & { 2}  & { 1}   & { 8640}      & { 17280}     & { 40320} \\
\hline
{ $\mathbb{A}_4^{\oplus 2}$}      & { 8}        & { 2}  & { 1}   & { 14400}     & { 57600}     & { 12096} \\
\hline
{ $\mathbb{E}_6+\mathbb{A}_2$}    & { 4}        & { 2}  & { 1}   & { 311040}    & { 622080}    & { 1120}  \\
\hline
{ $\mathbb{A}_7+\mathbb{A}_1$}    & { 2}        & { 1}  & { 1}   & { 80640}     & { 161280}    & { 4320}  \\
\hline
$\mathbb{D}_6+\mathbb{A}_2^{\oplus 2}$       & { 4}        & { 2}  & { 1}   & { 645120}    & { 1290240}   & { 540}   \\
\hline
{ $\mathbb{D}_5+\mathbb{A}_3$}    & { 4}        & { 2}  & { 1}   & { 46080}     & { 92160}     & { 7560}  \\
\hline
$\mathbb{D}_4^{\oplus 2}$         & { 72}       & { 6}  & { 1}   & { 36864}     & { 442368}    & { 1575}  \\
\hline
$(\mathbb{A}_3+\mathbb{A}_1)^{\oplus 2}$    & { 16}       & { 2}  & { 1}   & { 2304}      & { 18432}     & { 37800} \\
\hline
$\mathbb{A}_2^{\oplus 4}$         & { 384}      & { 8}  & { 1}   & { 1296}      & { 62208}     & { 11200} \\
\hline
{ $\mathbb{E}_8$}       & { 1}        & { 1}  & { 1}   & { 696729600} & { 696729600} & { 1} \\
\hline
\end{tabular}
    \label{LatticeNum}
    \caption{The number of lattices conjugate to $\Lambda$ in $\mathbb{E}_8$}
\end{table}
\newpage
\par Now, before we calculate $\textup{tor }\mathbb{Q}^*/\langle z_1\ldots z_9, \Lambda\rangle$, we shall first make the following observation: let $\Lambda'\subseteq\mathbb{E}_8$ be a lattice. Then $\Lambda\subset\Lambda'$ if and only if there exists a unique quotient epimorphism $\varphi:\mathbb{E}_8/\Lambda\twoheadrightarrow \mathbb{E}_8/\Lambda'$. The value of $\mathbb{E}_8/\Lambda$ is equal to the geometric Mordell-Weil group $E(K)$ in main theorem of \cite{OguisoShioda91} which is as follows:
\begin{table}[ht]
        \centering
        \begin{tabular}{ |c|c| } 
         \hline
         $\Lambda$ & $E(K)$\cite{OguisoShioda91} \\ 
         \hline
         $\mathbb{E}_8$ & $1$\\ 
         \hline
         $\mathbb{A}_8$ & $\mathbb{Z}/3$ \\ 
         \hline
         $\mathbb{D}_8$ & $\mathbb{Z}/2$\\ 
         \hline
         $\mathbb{E}_7\oplus\mathbb{A}_1$ & $\mathbb{Z}/2$ \\ 
         \hline
         $\mathbb{A}_5\oplus\mathbb{A}_2\oplus\mathbb{A}_1$ & $\mathbb{Z}/6$ \\ 
         \hline
         $\mathbb{A}_4^{\oplus 2}$ & $\mathbb{Z}/5$ \\ 
         \hline
         $\mathbb{A}_2^{\oplus 4}$ & $(\mathbb{Z}/3)^2$ \\ 
         \hline
         $\mathbb{E}_6\oplus\mathbb{A}_2$ & $\mathbb{Z}/3$ \\ 
         \hline
         $\mathbb{A}_7\oplus\mathbb{A}_1$ & $\mathbb{Z}/4$ \\ 
         \hline
         $\mathbb{D}_6\oplus\mathbb{A}_1^{\oplus 2}$ & $(\mathbb{Z}/2)^2$  \\ 
         \hline
         $\mathbb{D}_5\oplus\mathbb{A}_3$ & $\mathbb{Z}/4$ \\ 
         \hline
         $\mathbb{D}_4^{\oplus 2}$ & $(\mathbb{Z}/2)^2$ \\ 
         \hline
         $(\mathbb{A}_3\oplus\mathbb{A}_1)^{\oplus 2}$ & $\mathbb{Z}/4\oplus \mathbb{Z}/2$ \\ 
         \hline
        \end{tabular}
        \caption{\cite{OguisoShioda91}}
        \label{tab:E8tor_table}
    \end{table}
\par With this in mind, we may state the following:
\begin{lemma}
    Let $R$ be a rank 8 root lattice contained in $\mathbb{E}_8$ and $\Lambda$ a rank 8 sublattice of $\mathbb{E}_8$ which contains $R$. Then $\Lambda$ is a root lattice. Moreover, the containment of rank 8 root lattice is as follows:
    
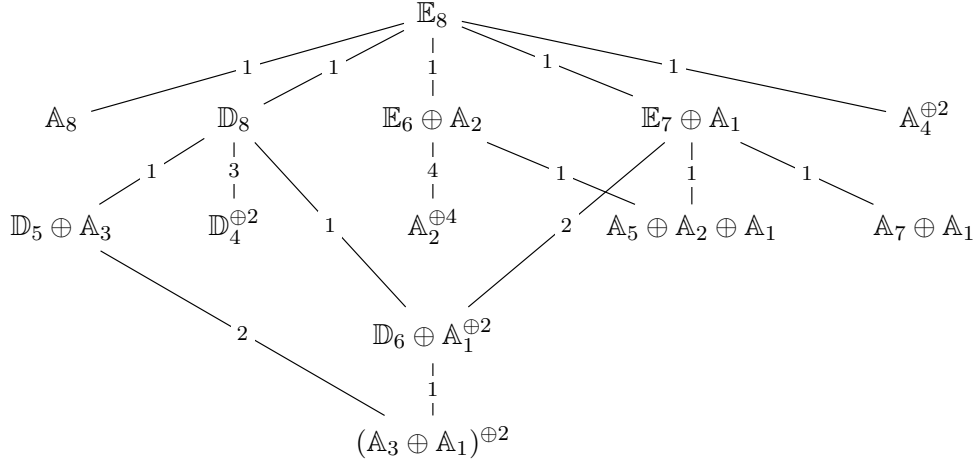
\begin{center}
        \begin{tikzcd}
                                                                      &                                                                                                                             & \mathbb{E}_8 \arrow[d, "1" description, no head] \arrow[rd, "1" description, no head] \arrow[rrd, "1" description, no head] \arrow[ld, "1" description, no head] \arrow[lld, "1" description, no head] &                                                                                                                                               &                                \\
\mathbb{A}_8                                                          & \mathbb{D}_8 \arrow[d, "3" description, no head] \arrow[ld, "1" description, no head] \arrow[rdd, "1" description, no head] & \mathbb{E}_6\oplus \mathbb{A}_2 \arrow[rd, "1" description, no head] \arrow[d, "4" description, no head]                                                                                               & \mathbb{E}_7\oplus\mathbb{A}_1 \arrow[d, "1" description, no head] \arrow[rd, "1" description, no head] \arrow[ldd, "2" description, no head] & \mathbb{A}_4^{\oplus 2}        \\
\mathbb{D}_5\oplus\mathbb{A}_3 \arrow[rrdd, "2" description, no head] & \mathbb{D}_4^{\oplus 2}                                                                                                     & \mathbb{A}_2^{\oplus 4}                                                                                                                                                                                & \mathbb{A}_5\oplus\mathbb{A}_2\oplus\mathbb{A}_1                                                                                              & \mathbb{A}_7\oplus\mathbb{A}_1 \\
                                                                      &                                                                                                                             & \mathbb{D}_6\oplus\mathbb{A}_1^{\oplus 2} \arrow[d, "1" description, no head]                                                                                                                          &                                                                                                                                               &                                \\
                                                                      &                                                                                                                             & (\mathbb{A}_3\oplus\mathbb{A}_1)^{\oplus 2}                                                                                                                                                            &                                                                                                                                               &                               
\end{tikzcd}
    \end{center}
    where each number is equal to the number of superlattices of the given type.
\end{lemma}\label{rootcontainment}
\begin{proof}
\par This trivially holds for $\mathbb{A}_8, \mathbb{D}_8, \mathbb{E}_6\oplus\mathbb{A}_2, \mathbb{E}_7\oplus\mathbb{A}_1, \mathbb{A}_4^{\oplus 2}$ as $\mathbb{E}_8/R$ are of prime order, and so we shall consider the rest of the lattices in cases. 
\par Case $\mathbb{A}_7\oplus\mathbb{A}_1, \mathbb{D}_5\oplus\mathbb{A}_3:$ As $\mathbb{E}_8/R\cong\mathbb{Z}/4$ for these root lattices have only one nontrivial subgroup, they have only one superlattice in $\mathbb{E}_8$, which we can show constructively, by showing the embedding of a root lattice into the affine Dynkin diagram of these lattices, to be a root lattice as in Figures \ref{fig:A7A1E7A1} and \ref{fig:D5A3D8}.

\par Case $\mathbb{A}_5\oplus\mathbb{A}_2\oplus\mathbb{A}_1:$ As $\mathbb{E}_8/\mathbb{A}_5\oplus\mathbb{A}_2\oplus\mathbb{A}_1\cong \mathbb{Z}/6, \mathbb{A}_5\oplus\mathbb{A}_2\oplus\mathbb{A}_1$ has only 2 superlattices in $\mathbb{E}_8$. Again this is shown constructively as in Figures \ref{fig:A5A2A1E7A1} and \ref{A5A2A1E6A2}

\par Case $\mathbb{A}_2^{\oplus 4}:$ As $\mathbb{E}_8/\mathbb{A}_2^{\oplus 4}\cong (\mathbb{Z}/3)^2, \mathbb{A}_2^{\oplus 4}$ has 4 superlattices in $\mathbb{E}_8$. We see that $\mathbb{A}_2^{\perp}=\mathbb{E}_6$ in $\mathbb{E}_8$, and so for any given copy of $\mathbb{A}_2^{\oplus 4}$ it is contained by 4 copies of $\mathbb{E}_6\oplus\mathbb{A}_2$. The containment of $\mathbb{A}_2^{\oplus 3}$ in $\mathbb{E}_6$ is shown in Figure \ref{4A2E6A2}.

\par Case $\mathbb{D}_6\oplus\mathbb{A}_1^{\oplus 2}:$ As $\mathbb{E}_8/(\mathbb{D}_6\oplus\mathbb{A}_1^{\oplus 2})\cong (\mathbb{Z}/2)^2, \mathbb{D}_6\oplus\mathbb{A}_1^{\oplus 2}$ is contained by 3 superlattices in $\mathbb{E}_8$. As $\mathbb{A}_1^{\perp}=\mathbb{E}_7$, any given copy of $\mathbb{D}_6\oplus\mathbb{A}_1^{\oplus 2}$ is contained by two copies of $\mathbb{E}_7\oplus\mathbb{A}_1$. Finally we can show constructively that $\mathbb{D}_6\oplus\mathbb{A}_1^{\oplus 2}$ is contained by $\mathbb{D}_8$ as in Figure \ref{D62A1D8}.

\par Case $\mathbb{D}_4^{\oplus 2}: \mathbb{D}_4^{\oplus 2}$ has 3 superlattices in $\mathbb{E}_8$. We see by a calculation that there exist 35 copies of $\mathbb{D}_4^{\oplus 2}$ in $\mathbb{D}_8$. Furthermore we see trivially the following equation:
\begin{equation}
    \#\{\Lambda\subset \mathbb{E}_8\mid \Lambda\cong \mathbb{D}_4^{\oplus 2} \}\#\{\Lambda\subset \mathbb{E}_8\mid \Lambda\cong \mathbb{D}_8,  \mathbb{D}_4^{\oplus 2}\subset \Lambda \} = \#\{\Lambda\subset \mathbb{E}_8\mid \Lambda\cong \mathbb{D}_8 \}
\end{equation}
Applying the values from Table \ref{LatticeNum} we see that all three superlattices are of type $\mathbb{D}_8$. 
\par Case $(\mathbb{A}_3\oplus\mathbb{A}_1)^{\oplus 2}:$ As $\mathbb{E}_8/(\mathbb{A}_3\oplus\mathbb{A}_1)^{\oplus 2}\cong \mathbb{Z}/4\oplus\mathbb{Z}/2$, we see that it is sufficient to show that all copies of $\mathbb{Z}/4$ and $(\mathbb{Z}/2)^2$ correspond to root lattices. We see that this lattice is contained by $\mathbb{D}_6\oplus\mathbb{A}_1^{\oplus 2}$ and $\mathbb{D}_5\oplus\mathbb{A}_3$ as in Figures \ref{2A32A1D62A1} and \ref{2A32A1D5A3}.

\par We see that $\mathbb{E}_8/(\mathbb{D}_6\oplus\mathbb{A}_1^{\oplus 2})\cong(\mathbb{Z}/2)^2$ and so there exists only one superlattice of type $\mathbb{D}_6\oplus\mathbb{A}_1^{\oplus 2}$, and finally since $\mathbb{A}_3^{\perp}=\mathbb{D}_5$ in $\mathbb{E}_8$ that there are two superlattices of $(\mathbb{A}_3\oplus\mathbb{A}_1)^{\oplus 2}$ of type $\mathbb{D}_5\oplus\mathbb{A}_3$. Since $\mathbb{E}_8/(\mathbb{D}_5\oplus\mathbb{A}_3)\cong\mathbb{Z}/4$, we have shown that all lattices corresponding to $\mathbb{Z}/4$ and $(\mathbb{Z}/2)^2$ are root lattices, which means that all subgroups of $\mathbb{Z}/4\oplus\mathbb{Z}/2$ correspond to root lattices. 
\par And so all rank 8 superlattices of a rank 8 root lattice in $\mathbb{E}_8$ are root lattices. 
\end{proof}
\begin{corollary}
    For some rank 8 lattice $\Lambda\subset\mathbb{E}_8$, let $\overline{\delta}(\Lambda)$ be the density of primes such that $\ker\overline{\textup{res}}_p = \Lambda$. Then the density of polyhedral primes p is
    \begin{equation}
    \delta(X) = \overline{\delta}(\mathbb{E}_8) + \sum_{\substack{\Lambda \subset \mathbb{E}_8\\ \Lambda \textup{ is a root lattice}\\ \textup{of rank 8}}} \overline{\delta}(\Lambda)
    \end{equation}
\end{corollary}

And so what we are left to calculate is $\textup{tor }\mathbb{Q}^*/\langle b_1\ldots b_9, \Lambda\rangle$. In order to do this, we may consider $\Pi/\Lambda$, where $\Pi= \langle z_1, \ldots, z_9\rangle/\langle z_1\ldots z_9\rangle$. Once found, as $\mathbb{Q}^*/\langle z_1\ldots z_9\rangle$ is torsion free, we can fix a basis of  $z_1\ldots z_9, \beta_1, \ldots, \beta_8$ of $\langle z_1\ldots z_9, \Lambda\rangle$ such that $\Pi/\Lambda\cong\mathbb{Z}/m_{\beta_1}\oplus\ldots\oplus\mathbb{Z}/m_{\beta_8}$ and $\langle \beta_1, \ldots, \beta_8\rangle\textup{ mod } z_1\ldots z_9 = \Lambda \textup{ mod } z_1\ldots z_9$. This means that $\delta(z_1\ldots z_9, \Lambda)=\delta(z_1\ldots z_9, \beta_1, \ldots, \beta_8)$ and $\overline{\delta}(z_1\ldots z_9, \Lambda)=\overline{\delta}(z_1\ldots z_9, \beta_1, \ldots, \beta_8)$. By a mechanical calculation we find that $\Pi/\mathbb{E}_8\cong \mathbb{Z}/3$. It follows that there exists some basis $z_1\ldots z_9, \beta_1, \ldots, \beta_9$ such that $m_{\beta_1}=3$ and $m_{\beta_i}=1$ for $i\neq 1$. By Theorem \ref{bigeqden}
$$
\overline{\delta}(\mathbb{E}_8) = \sum_{d_1\mid 3}\varphi(d_1)\left(S_{d_1, 1}+\sum_{\substack{x\in \widehat{a}V}}S_{[2^{C_x}, d_1], [2,\Delta(\widehat{x})]}+\sum_{\substack{x\in V/\langle \widehat{a}\rangle\\ x\neq 1}}S_{[2^{C_x}, d_1], \Delta(\widehat{x})} \right)
$$
and for all other rank 8 root lattices we may define it recursively:
\begin{equation}\label{sumrecursion}
    \overline{\delta}(\Lambda) = \delta(\Lambda) - \sum_{\substack{\Lambda\subset \Lambda'\subset \mathbb{E}_8\\ \Lambda' \textup{ is a root lattice}}}\overline{\delta}(\Lambda')
\end{equation}
\begin{lemma}
    Let $R$ be a rank 8 root system in $\mathbb{E}_8$ and $H$ a group containing a subgroup isomorphic to $\mathbb{Z}/3$. Then the set of sublattices $\Lambda\subset\mathbb{E}_8$ such that $\Lambda$ is of type $R$ and $\Pi/\Lambda\cong H$ is given by the following table:
    \begin{table}[!ht]
    \centering
    \begin{tabular}{|c|c|c|c|}
    \hline
        R & $\mathbb{E}_8/R$ & $H$ & $\#\{\Lambda\subset\mathbb{E}_8\mid \Lambda\cong R, \Pi/\Lambda\cong H\}$ \\ \hline
        $\mathbb{A}_8$ & $\mathbb{Z}/3$ & $\mathbb{Z}/9$ & 648 \\ \cline{3-4}
        ~ & ~ & $(\mathbb{Z}/3)^2$ & 312 \\ \hline
        $\mathbb{D}_8$ & $\mathbb{Z}/2$ & $\mathbb{Z}/6$ & 135 \\ \hline
        $\mathbb{E}_7\oplus\mathbb{A}_1$ & $\mathbb{Z}/2$ & $\mathbb{Z}/6$ & 120 \\ \hline
        $\mathbb{A}_5\oplus\mathbb{A}_2\oplus\mathbb{A}_1$ & $\mathbb{Z}/6$ & $\mathbb{Z}/18$ & 27216 \\ \cline{3-4}
        ~ & ~ & $\mathbb{Z}/6\oplus\mathbb{Z}/3$ & 13104 \\ \hline
        $\mathbb{A}_4^{\oplus 2}$ & $\mathbb{Z}/5$ & $\mathbb{Z}/15$ & 12096 \\ \hline
        $\mathbb{E}_6\oplus\mathbb{A}_2$ & $\mathbb{Z}/3$ & $\mathbb{Z}/9$ & 756 \\ \cline{3-4}
        ~  & ~ & $(\mathbb{Z}/3)^2$ & 364 \\ \hline
        $\mathbb{A}_7\oplus\mathbb{A}_1$ & $\mathbb{Z}/4$ & $\mathbb{Z}/12$ & 4320 \\ \hline
        $\mathbb{D}_6\oplus\mathbb{A}_1^{\oplus 2}$ & $(\mathbb{Z}/2)^2$ & $\mathbb{Z}/6\oplus\mathbb{Z}/2$ & 540 \\ \hline
        $\mathbb{D}_5\oplus\mathbb{A}_3$ & $\mathbb{Z}/4$ & $\mathbb{Z}/12$ & 7560 \\ \hline
        $\mathbb{D}_4^{\oplus 2}$ & $(\mathbb{Z}/2)^2$ & $\mathbb{Z}/6\oplus\mathbb{Z}/2$ & 1575 \\ \hline
        $(\mathbb{A}_3\oplus\mathbb{A}_1)^{\oplus 2}$ & $\mathbb{Z}/4\oplus\mathbb{Z}/2$ & $\mathbb{Z}/6\oplus\mathbb{Z}/4$ & 37800 \\ \hline
        $\mathbb{A}_2^{\oplus 4}$ & $(\mathbb{Z}/3)^2$ & $\mathbb{Z}/9\oplus\mathbb{Z}/3$ & 10080 \\ \cline{3-4}
        ~ & ~ & $(\mathbb{Z}/3)^3$ & 1120 \\ \hline
        $\mathbb{E}_8$ & $1$ & $\mathbb{Z}/3$ & 1 \\ \hline
    \end{tabular}
    \caption{The number of rank 8 lattices $\Lambda$ of type $R$ such that $\Pi/\Lambda\cong H$}
    \label{tab:PiLambTable}
\end{table}
\end{lemma}
\newpage
\begin{proof}
As $\Pi/\mathbb{E}_8\cong\mathbb{Z}/3$, if $\mathbb{E}_8/\Lambda$ has no subgroup of order 3, then $\Pi/\Lambda$ is invariant under the Weyl group transformation on $\Lambda$ and is equal to $\mathbb{Z}/3\oplus\mathbb{E}_8/\Lambda$. This leaves up 4 cases: $\mathbb{A}_8, \mathbb{A}_5\oplus\mathbb{A}_2\oplus\mathbb{A}_1, \mathbb{A}_2^{\oplus 4}, \mathbb{E}_6\oplus\mathbb{A}_2$. 
\par In the remaining cases, as $\mathbb{E}_8/\Lambda$ contains a copy of $\mathbb{Z}/3$, it is possible that $\Pi/\Lambda$ might contain a copy of $\mathbb{Z}/9$. If we choose a lattice $\Lambda$ of type $\mathbb{A}_2^{\oplus 4}$ in $\mathbb{E}_8$ such that $\mathbb{Z}/9\triangleleft\Pi/\Lambda$, we see that it must contain 3 superlattices $\Lambda_1, \Lambda_2, \Lambda_3$ in $\mathbb{E}_8$ such that $\Pi/\Lambda_i\cong \mathbb{Z}/9$ and one more superlattice $\Lambda_4$ in $\mathbb{E}_8$ such that $\Pi/\Lambda_4\cong (\mathbb{Z}/3)^2$. As these superlattices all correspond to lattices of type $\mathbb{E}_6\oplus\mathbb{A}_2$, and all copies of $\mathbb{E}_6\oplus\mathbb{A}_2$ contain a copy of $\mathbb{A}_2^{\oplus 4}$, we are able to relate the copies of $\mathbb{A}_2^{\oplus 4}$ which contain $\mathbb{Z}/9$ and the copies of $\mathbb{E}_6\oplus\mathbb{A}_2$ which contain $\mathbb{Z}/9$:
\begin{equation}
    \frac{3}{4}\frac{\#\{\Lambda\subset\mathbb{E}_8\mid \Lambda\cong\mathbb{A}_2^{\oplus 4}, \Pi/\Lambda\cong\mathbb{Z}/9\oplus\mathbb{Z}/3\}}{\#\{\Lambda\subset\mathbb{E}_8\mid \Lambda\cong\mathbb{A}_2^{\oplus 4}\}} = \frac{\#\{\Lambda\subset\mathbb{E}_8\mid \Lambda\cong\mathbb{E}_6\oplus\mathbb{A}_2, \Pi/\Lambda\cong\mathbb{Z}/9\}}{\#\{\Lambda\subset\mathbb{E}_8\mid \Lambda\cong\mathbb{E}_6\oplus\mathbb{A}_2\}}
\end{equation}
Likewise we see that if we choose a lattice $\Lambda$ of type $\mathbb{A}_5\oplus\mathbb{A}_2\oplus\mathbb{A}_1$ such that $\Pi/\Lambda\cong\mathbb{Z}/18$, we again see that there exist one lattice $\Lambda'$ in $\mathbb{E}_8$ such that $\Pi/\Lambda'\cong\mathbb{Z}/9$. As $\Lambda'$ is a copy of $\mathbb{E}_6\oplus\mathbb{A}_2$, we again are able to relate the copies of $\mathbb{A}_5\oplus\mathbb{A}_2\oplus\mathbb{A}_1$ which contain $\mathbb{Z}/9$ and the copies of $\mathbb{E}_6\oplus\mathbb{A}_2$ which contain $\mathbb{Z}/9$:
\begin{equation}
    \frac{\#\{\Lambda\subset\mathbb{E}_8\mid \Lambda\cong\mathbb{A}_5\oplus\mathbb{A}_2\oplus\mathbb{A}_1, \Pi/\Lambda\cong\mathbb{Z}/18\}}{\#\{\Lambda\subset\mathbb{E}_8\mid \Lambda\cong\mathbb{A}_5\oplus\mathbb{A}_2\oplus\mathbb{A}_1\}} = \frac{\#\{\Lambda\subset\mathbb{E}_8\mid \Lambda\cong\mathbb{E}_6\oplus\mathbb{A}_2, \Pi/\Lambda\cong\mathbb{Z}/9\}}{\#\{\Lambda\subset\mathbb{E}_8\mid \Lambda\cong\mathbb{E}_6\oplus\mathbb{A}_2\}}
\end{equation}
\par Finally, by a direct calculation we see that $\mathbb{E}_8/\Lambda\cong\mathbb{Z}$ for any $\Lambda\cong \mathbb{A}_5\oplus\mathbb{A}_2$ and so for any such $\Lambda$, we see that $\Pi/\Lambda$ can be either $\mathbb{Z}$ or $\mathbb{Z}/3\oplus\mathbb{Z}$. Now let $\Lambda_1\cong\mathbb{E}_6\oplus\mathbb{A}_2$ and $\Lambda_2\cong\mathbb{A}_8$ such that they both contain a given $\Lambda\cong\mathbb{E}_6\oplus\mathbb{A}_2$. As $\Lambda$ is of rank 7 and $\Lambda_1, \Lambda_2$ rank 8, it is clear that $\Pi/\Lambda_1\cong\mathbb{Z}/9$ if and only if $\Pi/\Lambda\cong\mathbb{Z}$ and $\Pi/\Lambda_2\cong\mathbb{Z}/9$ if and only if $\Pi/\Lambda\cong\mathbb{Z}$. This means as result that
\begin{equation}
    \frac{\#\{\Lambda\subset\mathbb{E}_8\mid \Lambda\cong\mathbb{A}_8, \Pi/\Lambda\cong\mathbb{Z}/9\}}{\#\{\Lambda\subset\mathbb{E}_8\mid \Lambda\cong\mathbb{A}_8\}} = \frac{\#\{\Lambda\subset\mathbb{E}_8\mid \Lambda\cong\mathbb{E}_6\oplus\mathbb{A}_2, \Pi/\Lambda\cong\mathbb{Z}/9\}}{\#\{\Lambda\subset\mathbb{E}_8\mid \Lambda\cong\mathbb{E}_6\oplus\mathbb{A}_2\}}
\end{equation}
So it suffices to calculate $\#\{\Lambda\subset\mathbb{E}_8\mid \Lambda\cong\mathbb{E}_6\oplus\mathbb{A}_2, \Pi/\Lambda\cong\mathbb{Z}/9\}$ in order to find all $\Lambda$ such that $\mathbb{Z}/9\triangleleft \Pi/\Lambda$, which we may do computationally. By a similar argument we also see that
\begin{equation}
    \frac{\#\{\Lambda\subset\mathbb{E}_8\mid \Lambda\cong\mathbb{E}_6\oplus\mathbb{A}_1, \Pi/\Lambda\cong\mathbb{Z}\}}{\#\{\Lambda\subset\mathbb{E}_8\mid \Lambda\cong\mathbb{E}_6\oplus\mathbb{A}_1\}} = \frac{\#\{\Lambda\subset\mathbb{E}_8\mid \Lambda\cong\mathbb{E}_6\oplus\mathbb{A}_2, \Pi/\Lambda\cong\mathbb{Z}/9\}}{\#\{\Lambda\subset\mathbb{E}_8\mid \Lambda\cong\mathbb{E}_6\oplus\mathbb{A}_2\}}
\end{equation}
This ratio can be found using the code in Appendix \ref{ratcode}. As a result, using values from Table \ref{LatticeNum}, we get the values in Table \ref{tab:PiLambTable}. 
\end{proof}
\begin{proof}[Proof of Theorem \ref{tablemainthm}]
\par By Corollary \ref{infsetup}, we see that $\Delta(\widehat{b_i})\rightarrow\infty$ that 
$$
\lim_{\Delta(\widehat{b_i})\rightarrow\infty, \forall i} \delta(a, b_1, \ldots, b_8) = \sum_{\substack{d\mid m_a\\ d_1\mid m_{b_1}\\ \cdots\\ d_8\mid m_{b_8}}}\Big(\varphi(d)\prod_{h=1}^8 \varphi(d_h)\Big)S^{(8)}_{[d_h]_h, d}.
$$
As a result we see that
$$
\lim_{\Delta(\widehat{b_i})\rightarrow\infty, \forall i} \overline{\delta}(\mathbb{E}_8) = \sum_{d_1\mid 3}\varphi(d_1)S^{(8)}_{d_1, 1} 
$$
and
$$
\lim_{\Delta(\widehat{b_i})\rightarrow\infty, \forall i}\delta(\Lambda) = \lim_{\Delta(\widehat{b_i})\rightarrow\infty, \forall i}\delta(z_1\ldots z_9, \Lambda) = \sum_{\substack{d_1\mid m_{\beta_1}\\ \cdots\\ d_8\mid m_{\beta_8}}}\Big(\prod_{h=1}^8 \varphi(d_h)\Big)S^{(8)}_{[d_h]_h, 1}
$$
where $\beta_1, \ldots, \beta_8$ denotes the basis of $\Lambda \textup{ mod }\langle a\rangle$ such that $\mathbb{E}_8/\Lambda\cong\mathbb{Z}/m_{\beta_1}\oplus\ldots\oplus\mathbb{Z}/m_{\beta_8}$. Now, using Equation \ref{sumrecursion}, either by a hand calculation or through the magma code in Appendix \ref{tordelcode} we find the values in Table \ref{tab:density_table}. 
\par This means that the density or primes such that $X_p$ is polyhedral for $\widehat{b_i}=b_i$ as $\Delta(\widehat{b_i})\rightarrow\infty$ is $83568208560360063877/43166735003229880320 S^{(8)}$.
\end{proof}

\appendix
\section{Diagrams for Theorem \ref{rootcontainment}}

In these figures we denote the regular nodes of the Dynkin with a shaded circle and the added node in the affine Dynkin diagram with a $\times$.
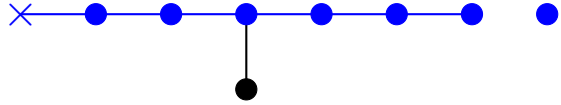
\begin{figure}[!ht]
    \centering
    \begin{tikzpicture}
    \draw[blue, thick] (-4, 1) -- (2, 1);
    \draw[black, thick] (-1, 1) -- (-1, 0);
    \filldraw[black] (-1, 0) circle (4pt);
    \filldraw[blue] (-1, 1) circle (4pt);
    \filldraw[blue] (-2, 1) circle (4pt);
    \filldraw[blue] (-3, 1) circle (4pt);
    \filldraw[blue] (0, 1) circle (4pt);
    \filldraw[blue] (1, 1) circle (4pt);
    \filldraw[blue] (2, 1) circle (4pt);
    \filldraw[blue] (3, 1) circle (4pt);
    \draw[blue, thick] (-4,1) node {$\bigtimes$};
\end{tikzpicture}
    \caption{$\mathbb{A}_7\oplus\mathbb{A}_1$ contained in $\mathbb{E}_7\oplus\mathbb{A}_1$}
    \label{fig:A7A1E7A1}
\end{figure}
\begin{figure}[!ht]
    \centering
    \begin{tikzpicture}
    \draw[blue, thick] (-2, 1) -- (1, 1);
    \draw[black,thick] (1, 1) -- (3, 1);
    \draw[blue,thick] (3, 1) -- (4, 1);
    \draw[blue,thick] (3, 1) -- (3, 2);
    \draw[blue, thick] (-1, 1) -- (-1, 0);
    \filldraw[blue] (-1, 0) circle (4pt);
    \filldraw[blue] (-1, 1) circle (4pt);
    \filldraw[blue] (-2, 1) circle (4pt);
    \filldraw[blue] (0, 1) circle (4pt);
    \filldraw[blue] (1, 1) circle (4pt);
    \filldraw[black] (2, 1) circle (4pt);
    \filldraw[blue] (3, 1) circle (4pt);
    \filldraw[blue] (4, 1) circle (4pt);
    \draw[blue, thick] (3,2) node {$\bigtimes$};
\end{tikzpicture}
    \caption{$\mathbb{D}_5\oplus\mathbb{A}_3$ contained in $\mathbb{D}_8$}
    \label{fig:D5A3D8}
\end{figure}
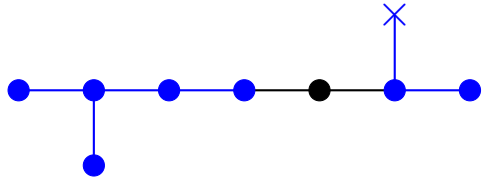
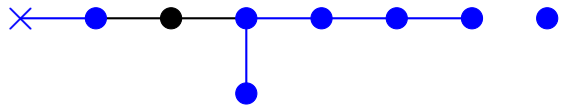
\begin{figure}[!ht]
    \centering
    \begin{tikzpicture}
    \draw[blue, thick] (-1, 1) -- (2, 1);
    \draw[blue, thick] (-4, 1) -- (-3, 1);
    \draw[black, thick] (-3, 1) -- (-1, 1);
    \draw[blue, thick] (-1, 1) -- (-1, 0);
    \filldraw[blue] (-1, 0) circle (4pt);
    \filldraw[blue] (-1, 1) circle (4pt);
    \filldraw[black] (-2, 1) circle (4pt);
    \filldraw[blue] (-3, 1) circle (4pt);
    \filldraw[blue] (0, 1) circle (4pt);
    \filldraw[blue] (1, 1) circle (4pt);
    \filldraw[blue] (2, 1) circle (4pt);
    \filldraw[blue] (3, 1) circle (4pt);
    \draw[blue, thick] (-4,1) node {$\bigtimes$};
\end{tikzpicture}
    \caption{$\mathbb{A}_5\oplus\mathbb{A}_2\oplus\mathbb{A}_1$ contained in $\mathbb{E}_7\oplus\mathbb{A}_1$}
    \label{fig:A5A2A1E7A1}
\end{figure} 
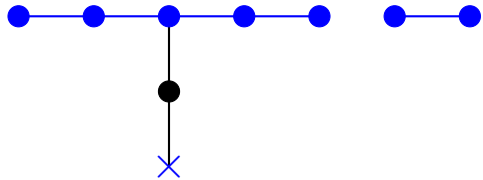
\begin{figure}[!ht]
    \centering
    \begin{tikzpicture}
    \draw[blue, thick] (0, 1) -- (1, 1);
    \draw[blue, thick] (-1, 1) -- (0, 1);
    \draw[blue, thick] (2, 1) -- (3, 1);
    \draw[blue, thick] (-3, 1) -- (-2, 1);
    \draw[blue, thick] (-2, 1) -- (-1, 1);
    \draw[black, thick] (-1, 1) -- (-1, 0);
    \draw[black, thick] (-1, 0) -- (-1, -1);
    \filldraw[black] (-1, 0) circle (4pt);
    \filldraw[blue] (-1, 1) circle (4pt);
    \filldraw[blue] (-2, 1) circle (4pt);
    \filldraw[blue] (-3, 1) circle (4pt);
    \filldraw[blue] (0, 1) circle (4pt);
    \filldraw[blue] (1, 1) circle (4pt);
    \filldraw[blue] (2, 1) circle (4pt);
    \filldraw[blue] (3, 1) circle (4pt);
    \draw[blue, thick] (-1,-1) node {$\bigtimes$};
\end{tikzpicture}
    \caption{$\mathbb{A}_5\oplus\mathbb{A}_2\oplus\mathbb{A}_1$ contained in $\mathbb{E}_6\oplus\mathbb{A}_2$}
    \label{A5A2A1E6A2}
\end{figure}
\begin{figure}[!ht]
    \centering
    \begin{tikzpicture}
    \draw[blue, thick] (0, 1) -- (1, 1);
    \draw[black, thick] (-1, 1) -- (0, 1);
    \draw[blue, thick] (2, 1) -- (3, 1);
    \draw[blue, thick] (-3, 1) -- (-2, 1);
    \draw[black, thick] (-2, 1) -- (-1, 1);
    \draw[black, thick] (-1, 1) -- (-1, 0);
    \draw[blue, thick] (-1, 0) -- (-1, -1);
    \filldraw[blue] (-1, 0) circle (4pt);
    \filldraw[black] (-1, 1) circle (4pt);
    \filldraw[blue] (-2, 1) circle (4pt);
    \filldraw[blue] (-3, 1) circle (4pt);
    \filldraw[blue] (0, 1) circle (4pt);
    \filldraw[blue] (1, 1) circle (4pt);
    \filldraw[blue] (2, 1) circle (4pt);
    \filldraw[blue] (3, 1) circle (4pt);
    \draw[blue, thick] (-1,-1) node {$\bigtimes$};
\end{tikzpicture}
    \caption{$\mathbb{A}_2^{\oplus 4}$ contained in $\mathbb{E}_6\oplus\mathbb{A}_2$}
    \label{4A2E6A2}
\end{figure}
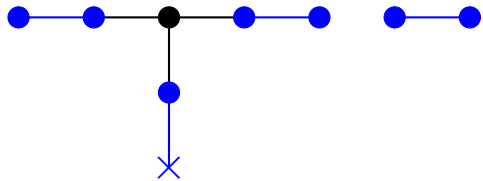
\begin{figure}[!ht]
    \centering
    \begin{tikzpicture}
    \draw[blue, thick] (-2, 1) -- (2, 1);
    \draw[black,thick] (2, 1) -- (3, 1);
    \draw[black,thick] (3, 1) -- (4, 1);
    \draw[black,thick] (3, 1) -- (3, 2);
    \draw[blue, thick] (-1, 1) -- (-1, 0);
    \filldraw[blue] (-1, 0) circle (4pt);
    \filldraw[blue] (-1, 1) circle (4pt);
    \filldraw[blue] (-2, 1) circle (4pt);
    \filldraw[blue] (0, 1) circle (4pt);
    \filldraw[blue] (1, 1) circle (4pt);
    \filldraw[blue] (2, 1) circle (4pt);
    \filldraw[black] (3, 1) circle (4pt);
    \filldraw[blue] (4, 1) circle (4pt);
    \draw[blue, thick] (3,2) node {$\bigtimes$};
\end{tikzpicture}
    \caption{$\mathbb{D}_6\oplus\mathbb{A}_1^{\oplus 2}$ contained in $\mathbb{D}_8$}
    \label{D62A1D8}
\end{figure}
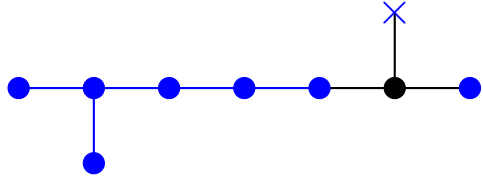
\begin{figure}[!ht]
    \centering
    \begin{tikzpicture}
    \draw[black, thick] (-1, 1) -- (1, 1);
    \draw[blue, thick] (-2, 1) -- (-1, 1);
    \draw[blue, thick] (-1, 1) -- (-1, 0);
    \draw[blue, thick] (1, 1) -- (2, 1);
    \draw[blue, thick] (1, 1) -- (1, 2);
    \filldraw[blue] (-1, 0) circle (4pt);
    \filldraw[blue] (-1, 1) circle (4pt);
    \filldraw[blue] (-2, 1) circle (4pt);
    \filldraw[black] (0, 1) circle (4pt);
    \filldraw[blue] (1, 1) circle (4pt);
    \filldraw[blue] (2, 1) circle (4pt);
    \filldraw[blue] (3, 1) circle (4pt);
    \filldraw[blue] (4, 1) circle (4pt);
    \draw[blue, thick] (1,2) node {$\bigtimes$};
\end{tikzpicture}
    \caption{$(\mathbb{A}_3\oplus\mathbb{A}_1)^{\oplus 2}$ contained in $\mathbb{D}_6\oplus\mathbb{A}_1^{\oplus 2}$}
    \label{2A32A1D62A1}
\end{figure}
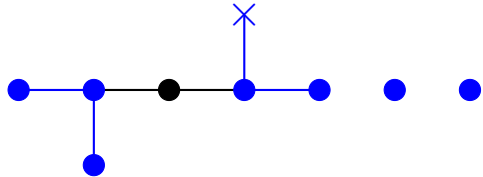
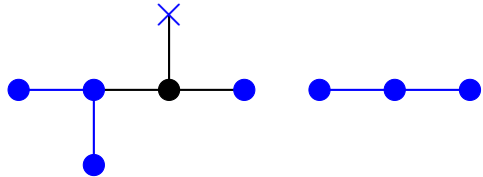
\begin{figure}[!ht]
    \centering
    \begin{tikzpicture}
    \draw[black, thick] (-1, 1) -- (1, 1);
    \draw[blue, thick] (-2, 1) -- (-1, 1);
    \draw[blue, thick] (-1, 1) -- (-1, 0);
    \draw[blue, thick] (2, 1) -- (4, 1);
    \draw[black, thick] (0, 1) -- (0, 2);
    \filldraw[blue] (-1, 0) circle (4pt);
    \filldraw[blue] (-1, 1) circle (4pt);
    \filldraw[blue] (-2, 1) circle (4pt);
    \filldraw[black] (0, 1) circle (4pt);
    \filldraw[blue] (1, 1) circle (4pt);
    \filldraw[blue] (2, 1) circle (4pt);
    \filldraw[blue] (3, 1) circle (4pt);
    \filldraw[blue] (4, 1) circle (4pt);
    \draw[blue, thick] (0,2) node {$\bigtimes$};
\end{tikzpicture}
    \caption{$(\mathbb{A}_3\oplus\mathbb{A}_1)^{\oplus 2}$ contained in $\mathbb{D}_5\oplus\mathbb{A}_3$}
    \label{2A32A1D5A3}
\end{figure}
\newpage
\section{Magma Code for Sections 3, 4}
\subsection{Code to Approximate Generalized Stephens's Constant}\label{stephconstcode}:
\begin{lstlisting}
RR := RealField();

//Adjusting up to which primes one calculates this for will, as expected, give a more accurate value
primes := PrimesUpTo(10^7);

//This code gives us the n-variable Stephens's Constant. In this case it will give is the 8-variable one. 
n:=8;

list := [RR!(1-(p^n-1)/(p-1)*p/(p^(n+2)-1)): p in primes];

C := 1;

for l in list do
	C*:= l;
end for;

print C;
\end{lstlisting}
\subsection{Code to Calculate $\frac{\#\{\Lambda\subset\mathbb{E}_8\mid \Lambda\cong\mathbb{E}_6\oplus\mathbb{A}_1, \Pi/\Lambda\cong\mathbb{Z}\}}{\#\{\Lambda\subset\mathbb{E}_8\mid \Lambda\cong\mathbb{E}_6\oplus\mathbb{A}_1\}}$}\label{ratcode}:
\begin{lstlisting}
G := FreeAbelianGroup(9);

//Defines the E8 Root System and its Weyl Group
R1:= RootSystem("E8");
W:= CoxeterGroup(GrpMat,"E8");
R1Roots := Roots(R1);

//Defines one copy of E6+A1 in E8
R3:= sub< R1| [1, 2, 3, 4, 5, 6, 8]>;

//Returns all copies of E6+A1 in E8 as a set of roots
O := Orbit(W, Roots(R4));
OList := [x: x in O];

//Converts the set of roots into its corresponding Root System
RList := [sub<R1| [Index(R1Roots, x[1]), Index(R1Roots, x[2]), Index(R1Roots, x[3]), Index(R1Roots, x[4]), Index(R1Roots, x[5]), Index(R1Roots, x[6]), Index(R1Roots, x[7])]> : x in OList];
//Gets the simple roots of each copy of E6+A1
SList := [SimpleRoots(r): r in RList];

//As all of the roots are written with the roots of E8 as the basis, this function converts it to values in Z8
CoordToRel := function(l) 
    c := [G.1-G.2, G.2-G.3, G.3-G.4, G.1+G.2+G.3, G.4-G.5, G.5-G.6, G.6-G.7, G.7-G.8];
    out := (Integers() ! l[1])*c[1];
    for i in [2 .. 8] do
        out := out+(Integers()! l[i])*c[i];
    end for;
    return out;
end function;
//Having obtained the simple roots for any given lattice L of rank n, this calculates Z9/<(1,...,1), L>
RootsToQuo := function(sR, n)
    rel := [CoordToRel(sR[i]) : i in [1 .. n]];
    return quo< G| Append(rel, G.1+G.2+G.3+G.4+G.5+G.6+G.7+G.8+G.9)>;
end function;

//Calculates Z9/<(1,...,1), L> for all L of type E6+A1
GpList:= [RootsToQuo(x, 7): x in SList];

//As all elements of GpList has either one or two generators, for any x in GpList, the element corresponding to x is 1 if it has 2 generators and 0 if it has one generator
GenList := [Ngens(x)-1: x in GpList];

//This calculates the ratio of L such that Z9/<(1,...,1), L>=Z/3+Z/3, and gives us the desired ratio
print (Rationals() ! Z9type/#GenList);

\end{lstlisting}

\subsection{Code to Calculate $\overline{\delta}(X)/S^{(8)}$}\label{tordelcode}:
\begin{lstlisting}
GrpSOne := function(G)
    genOrd := [ Order(g) : g in  Generators(G)];
    genDiv := [ Divisors(s): s in genOrd ];
    sum := 0;
    if #genOrd eq 1 then
        for d in Divisors(genOrd[1]) do
            prod := EulerPhi(d)/d^10;
            for p in PrimeDivisors(d) do
                prod *:= (p^9*(p^2-1))/(p^11-p^10-p^9+1);
            end for;
            sum +:= prod;
        end for;
    end if;
    if #genOrd eq 2 then
        for c in Divisors(genOrd[1]) do
            for d in Divisors(genOrd[2]) do
                prod := EulerPhi(c)*EulerPhi(d)/Lcm(c, d)^10;
                for p in PrimeDivisors(Lcm(c, d)) do
                prod *:= (p^9*(p^2-1))/(p^11-p^10-p^9+1);
                end for;
                sum +:= prod;
            end for;
        end for;
    end if;
    return sum;
end function;
GrpS := function(G, H)
    genOrd := [ Order(g) : g in  Generators(G)] cat [ Order(g) : g in  Generators(H)];
    genDiv := [ Divisors(s): s in genOrd ];
    sum := 0;
    if #genOrd eq 1 then
        for d in Divisors(genOrd[1]) do
            prod := EulerPhi(d)/d^10;
            for p in PrimeDivisors(d) do
                prod *:= (p^9*(p^2-1))/(p^11-p^10-p^9+1);
            end for;
            sum +:= prod;
        end for;
    end if;
    if #genOrd eq 2 then
        for c in Divisors(genOrd[1]) do
            for d in Divisors(genOrd[2]) do
                prod := EulerPhi(c)*EulerPhi(d)/Lcm(c, d)^10;
                for p in PrimeDivisors(Lcm(c, d)) do
                prod *:= (p^9*(p^2-1))/(p^11-p^10-p^9+1);
                end for;
                sum +:= prod;
            end for;
        end for;
    end if;
    if #genOrd eq 3 then
        for c in Divisors(genOrd[1]) do
            for d in Divisors(genOrd[2]) do
                for e in Divisors(genOrd[3]) do 
                    prod := EulerPhi(c)*EulerPhi(d)*EulerPhi(e)/Lcm([c, d, e])^10;
                    for p in PrimeDivisors(Lcm([c, d, e])) do
                    prod *:= (p^9*(p^2-1))/(p^11-p^10-p^9+1);
                    end for;
                    sum +:= prod;   
                end for;
            end for;
        end for;
    end if;
    return sum;
end function;
\end{lstlisting}

\printbibliography

\end{document}